\documentclass[11pt, review, authoryear]{elsarticle}
\usepackage{natbib}

\usepackage{amsthm}
\usepackage{amsmath}
\usepackage{amssymb}
\usepackage{url}
\usepackage[utf8]{inputenc}
\usepackage[english]{babel}
\usepackage[dvipsnames,table,longtable]{xcolor}
\usepackage{booktabs}
\usepackage{mathtools}
\usepackage{subcaption}
\usepackage[title]{appendix}
\usepackage{enumitem}
\usepackage{fix-cm}

\usepackage{accents}

\usepackage{changepage}
\usepackage{array}
\usepackage{float}
\usepackage[colorlinks=true, linkcolor=blue, citecolor=blue]{hyperref}
\usepackage{xr-hyper}
\usepackage{cleveref}

\usepackage{comment}
\usepackage{booktabs}
\usepackage{multicol}
\usepackage{multirow}
\usepackage{siunitx}

\usepackage{tikz}
\usetikzlibrary{positioning}
\usetikzlibrary{arrows.meta}
\usetikzlibrary{backgrounds}
\usepackage{centernot}

\usepackage{geometry}
\usepackage[normalem]{ulem}

\newcolumntype{C}[1]{>{\centering\arraybackslash}p{#1}}

\def \length {B}

\def \PRWA {RWA-P}
\def \path {\mathcal{P}}

\def \requestInd {r}
\def \requestSet {R}

\newtheorem{theorem}{Theorem}
\newtheorem{proposition}{Proposition}
\newtheorem{corollary}{Corollary}
\newtheorem{lemma}{Lemma}
\newtheorem{claim}{Claim}
\newtheorem{example}{Example}
\newtheorem{definition}{Definition}

\newcommand{\crev}{\color{black}} %{\color{red}} 
\newcommand{\cemph}{\color{black}}  %{\color{teal!80!black}} %{\color{magenta}} %%to emphasize the changes by oylum 

\newcommand{\bigO}{\mathcal{O}}

\newcommand{\IPbase}{IP\textsuperscript{Base}}
\newcommand{\IPstrong}{IP\textsuperscript{Strong}}

\newcommand{\penaltyCoef}{\rho}
\newcommand{\conflictSet}{\mathcal{C}}

\newcommand{\link}{e}
\newcommand{\linkSet}{E}
\newcommand{\linksOfLightpath}{\linkSet}%{\mathcal{L}}

\newcommand{\workingInd}{w}
\newcommand{\workingSet}{W}
\newcommand{\protectionInd}{p}
\newcommand{\protectionSet}{P}
\newcommand{\wavelength}{\lambda}
\newcommand{\wavelengthSet}{\Lambda}
\newcommand{\wavelengthOfLightpath}{\Lambda}
\newcommand{\obj}{f}
\newcommand{\objFnc}{\obj(x,y)}

\newcommand{\objFncLink}{\obj_{\alpha}(x,y)}
\newcommand{\objFncGranted}{\obj_{\beta}(x)}

\newcommand{\betaBase}{\beta^{\text{Base}}}
\newcommand{\betaTight}{\beta^{\text{Tight}}}
\newcommand{\LBprioOne}{\Omega^{=}}
\newcommand{\LBprioExact}{\Omega^{>}}

\newcommand{\Methods}{Methods}
\newcommand{\methods}{methods}

\newcommand{\method}{method}

\newcommand{\mssnode}{n}
\newcommand{\wnode}{u}
\newcommand{\pnode}{v}

\newcommand{\exMSSinst}{G_s}
\newcommand{\exPRWAnetwork}{G}
\newcommand{\exConfSet}{\conflictSet}

\newcommand{\RSheur}{RS-Heur}

\newcommand{\mss}{{\sc MSS}}
\newcommand{\dmss}{{{\sc Dec-}\mss}}

\newcommand{\PRWAreqInst}{\mathcal{I}}

\newcommand{\PRWAreq}{\PRWA-\sc r}
\newcommand{\DPRWAreq}{{\sc Dec-}\PRWAreq}
\newcommand{\cpink}{\color{magenta}} %to emphasize some of the changes that oylum has made from May 31 onward

\newcommand{\BI}{\begin{itemize}}
\newcommand{\EI}{\end{itemize}}
\newcommand{\I}{\item}

\def\overbigdot#1{\overset{\hbox{\tiny$\bullet$}}{#1}}
\newcommand{\xgiven}{\overbigdot{x}}
\newcommand{\ygiven}{\overbigdot{y}}
\newcommand{\fgiven}{\overbigdot{f}}

\makeatletter
\newcommand*{\addFileDependency}[1]{% argument=file name and extension
  \typeout{(#1)}
  \@addtofilelist{#1}
  \IfFileExists{#1}{}{\typeout{No file #1.}}
}
\makeatother

\begin{document}

\begin{frontmatter}
% Enter the full title:
\title{Routing and Wavelength Assignment with Protection: A Quadratic Unconstrained Binary Optimization Approach}
%A QUBO and Digital Annealer Approach -- TITLE CANNOT CONTAIN ACRONYMS IN EJOR!!!
	
\author[1]{Oylum \c{S}eker\corref{cor1}}
\ead{oylum.seker@utoronto.ca}
	
\author[1]{Merve Bodur}
\ead{bodur@mie.utoronto.ca}
		
\author[1]{Hamed Pouya}	
\ead{h.pouya@utoronto.ca}
	
\cortext[cor1]{Corresponding author}
	
\address[1]{Department of Mechanical and Industrial Engineering, University of Toronto, Toronto, ON, M5S3G8, Canada}

\begin{abstract}
The routing and wavelength assignment with protection is an important problem in telecommunications. 
Given an optical network and incoming connection requests, a commonly studied variant of the problem aims to grant maximum number of requests by assigning lightpaths at minimum network resource usage level, while ensuring the provided services remain functional in case of a single-link failure through dedicated path protection. 
We consider a practically relevant version where alternative lightpaths for requests are assumed to be given as a precomputed set, and show that it is NP-hard. 
We formulate the problem as an integer programming (IP) model, and also use it as a foundation to develop a novel quadratic unconstrained binary optimization (QUBO) model,  which can be both directly solved by a state-of-the-art solver like GUROBI.
We present necessary and sufficient conditions on objective function parameters to prioritize request granting objective over wavelength-link usage for both models, and a sufficient condition to ensure the exactness of the QUBO model. 
Moreover, we implement a problem-specific branch-and-cut algorithm for the IP model, and employ a new quantum-inspired technology, Digital Annealer (DA), for the QUBO model. 
%we present conditions on model parameters to achieve a desired objective prioritization, and to ensure the exactness of the QUBO model. 
%{\cemph %\sout{We employ state-of-the-art solver GUROBI to solve the IP models by providing them directly to the solver and also applying a branch-and-cut method. For solving the QUBO models, we use a promising new technology, Digital Annealer (DA), as well as GUROBI.}}
%{ compare it with a problem-specific heuristic and three different exact alternatives that employ GUROBI, namely providing the models directly to the solver, and applying a branch-and-cut method. }
We conduct computational experiments on a large suite of instances {\cemph that are hard to optimally solve} in order to assess the efficiency and efficacy of all of these approaches as well as a problem-specific heuristic. 
%for performance {\cemph assessment} \sout{comparison} and also do sensitivity analysis on {\cemph penalty coefficient parameter.} 
%\sout{model parameters. } 
The results show that the emerging technology DA outperforms the considered established techniques coupled with GUROBI, in finding mostly significantly better or as good solutions in only two minutes compared to two hours of run time, whereas the problem-specific heuristic fails to be competitive. 
%{\cemph We find that the solution quality that the emerging technology DA yields in two minutes is mostly better than or as good as what the established techniques coupled with the state-of-the-art solvers yield after two hours of run time.}
% penalty 
\end{abstract}

\end{frontmatter}

%\maketitle

\newpage	
\section{Introduction}
\label{intro}
% OS: introductory info about optical networks
An optical network is a medium for information transmission via signals encoded in pulses of light. It connects devices that generate or store data through optical fibers carrying light channels.
With the wavelength division multiplexing (WDM) technology allowing multiple %optical 
signals to be simultaneously transmitted on the same fiber, optical networks have become particularly potent in conveying high volumes of information at very high speeds in a reliable way. As such, they are being increasingly deployed to meet the rapidly growing demand in many high-bandwidth applications such as real-time multimedia streaming, cloud computing and mobile network services \citep{majumdar2018optical,chadha2019optical}. 

% OS: network model and lightpath definition
Wavelength-routed networks form a broad class of WDM networks, and can be considered as a set of nodes joined by fiber links. The communication between a pair of nodes is established through lightpaths, which is referred to as \emph{connecting} these two nodes.
A \emph{lightpath} is an optical communication channel between two nodes in the network, and is comprised of a path (route)
and a wavelength. 
% OS: define RWA
A typical problem arising in wavelength-routed networks is the {\it routing and wavelength assignment} (RWA) problem. Given a set of connection requests between pairs of nodes, the RWA problem decides which requests to \emph{grant}, i.e., provision a lightpath, such that no two allocated lightpaths with the same wavelength traverse a common fiber link, to prevent interference.

There are many variants of the RWA problem, which mainly differ in their %network characteristics,
objective and the nature of the connection demand process, as well as extensions incorporating additional concerns such as failures in the network, service quality and resource usage profile \citep{bandyopadhyay2007dissemination}. 
In general, they are all difficult to solve due to the inherent computational complexity of the RWA problem \citep{erlebach2001complexity}. 

In this paper, we study an RWA problem whose features are motivated by practical concerns (e.g., fast recovery from failures), 
%and short connection setup times), 
and some common goals in telecommunications industry (such as granted request maximization and resource usage minimization). 
In the sequel, we first provide all the relevant background information and motivate our problem setup (Section \ref{subsec:background}); then formally define our problem, introduce our solution approach and summarize our contributions (Section \ref{subsec:ourwork}). 

\subsection{Background Information}
\label{subsec:background}
The information provided in this section is mostly based on the books by \cite{mukherjee2006optical}, \cite{bandyopadhyay2007dissemination} and \cite{chatterjee2016routing}.

\subsubsection*{Network failures and recovery schemes.}
% OS: importance of network failures and how to deal with them
While it is desirable to maximize the number of granted requests in RWA, it is also often necessary to provide some degree of protection for them against potential failures in the network.
In a WDM network, any of the components may fail, and one failure may disrupt multiple connections. 
Link failure is the most frequently encountered type of fault, which may for instance arise from fiber cuts due to human errors during construction operations, or natural calamity such as earthquakes. 
% OS: single-link failures being the most common 
Also, the probability of having multiple link failures simultaneously, or having additional failure(s) before one has been repaired, is considered to be negligible. Thus, the literature is mostly concerned with the case of a \emph{single-link failure}, which we follow as well.

As each fiber link can carry terabits of data per second, even a brief disruption of a connection can result in a large amount of data loss. As such, fault management mechanisms play an important role in WDM network survivability \citep{zhang2004review}.
Link failures can be handled at the optical or a higher layer, but the time to detect and repair a failed link ranges from tens of seconds to a few days. Since even the shortest repair time is still too long relative to the rate of data transfer, some \emph{fault recovery strategies}, which re-route the broken connections using the available network resources, are usually adopted to keep the services functional while the repair process is in progress. 
% OS: fault recovery schemes; protection vs. restoration
There are two main categories of such strategies, {\it protection} and {\it restoration}.
In the protection scheme, \emph{backup (protection) lightpaths} are computed and reserved in advance along with the \emph{primary (working) lightpaths}. 
This ensures fast recovery of all the affected connections by replacing the primary lightpaths with the backups (in milliseconds upon failure), at the expense of increased network resource usage.
In the restoration scheme, on the other hand, backup capacity is not provisioned prior to the occurrence of failure, instead new lightpaths are discovered dynamically upon the interruption of connections. Despite being less demanding in resource usage, restoration schemes fail to guarantee resource availability, and lead to higher recovery times.

% OS: recovery schemes further classified; path vs link based approaches, dedicated
Protection and restoration schemes for link failures can be categorized with respect to being {\it path}- or {\it link}-based, i.e., whether they re-route the whole path or only the failed link(s). 
Path protection schemes can be {\it dedicated} or {\it shared}.
When it is dedicated, each backup lightpath is reserved for only one primary lightpath, and two backup lightpaths with the same wavelength cannot have a common link on their routes. Two commonly used subcategories of this scheme are denoted by 1:1 and 1+1. The former transmits data through only primary lightpaths before failure, while the latter allows the use of both primary and backup lightpaths simultaneously.

\subsubsection*{Static and dynamic RWA.}
% OS: static and dynamic demand definitions and the batched demand lying in btw these
The underlying demand process in RWA problems is considered as {\it static} or {\it dynamic}. 
The static case assumes that connection requests with their associated source and destination nodes are known in advance. In the dynamic case, on the other hand, requests arrive one by one and are provisioned a lightpath in real-time.
%In our study, we strive to develop efficient solution methods for the static case.
%In our study, we focus on the static case, but since the solution time of our proposed methodology is typically few minutes by design  good solutions in a short amount of time 
In our study, we aim to efficiently solve the static RWA problem.

We note that a dynamic RWA problem can be approximated by discretizing the time horizon into intervals and solving a static RWA problem for each using the batch of requests arrived during the corresponding interval, as typically done for dynamic network reconfiguration problems \citep{ zhang2007optimization, wu2012forward, grover2013distributed}. 
Such an approach would usually require solving the static problems quickly, e.g., for some bandwidth on demand services, their service level agreement between the provider and the customer guarantees connections to be established (with a certain level of protection) in a matter of minutes \citep{fawaz2004service,losego2005time}. 
In that regard, the methods solving the static RWA problem efficiently may well benefit dynamic settings. In fact, one method we propose can generate good-quality solutions in up to two minutes.
%our proposed methods may well benefit dynamic settings. %addressing the dynamic provisioning problem.

%{\cred Static RWA problems can be used as an approximation of the dynamic ones. The idea is to discretize the continuous time horizon into intervals/cycles, and solve a static RWA problem for each, using the batch of demand consisting of dynamically arriving requests during that interval as the input. This is indeed a typical setting in studies addressing network reconfiguration problems that establish new connections by allowing the existing ones to be rearranged \citep{ zhang2007optimization, wu2012forward, grover2013distributed}. It can also lead to more efficient use of potentially limited network resources in the long run. In that regard, a method that can yield high-quality solutions quickly to the static RWA problem, which our study strives to provide, may well help in addressing the practical provisioning problem that is dynamic by nature. }

\subsubsection*{Precomputed paths.}%Connection setup times and 
% OS: common objectives and connection setup times of RWA models

% Motivation for precomputation
One way to improve solution times for RWA problems is to employ a two-phase framework; generating a set of paths between all or some potential source and destination pairs in the first phase, and picking one from the precomputed alternatives for each request and doing the wavelength assignment in the second phase, for instance as adopted in \citep{li2000partition, noronha2006routing}. In cases where an RWA problem needs to be solved repeatedly over time, this strategy can be made more efficient by performing the first phase only once at the beginning, i.e., by using the same set of precomputed paths throughout the horizon. Granting requests from their precomputed set of alternatives may also provide more control to decision makers, in the sense that they can disperse the demand over the network in a balanced way if desired, and  can have a better idea on the state of the network for subsequent decision-making stages well in advance.

%{\cred When combined with a protection scheme (especially a dedicated one), using precomputed paths serves well to the joint purpose of quickly reacting to connection requests and recovering from failures. This combination is used in our problem setting, which we explain in more detail next.}

%The emerging high-bandwidth applications such as video-on-demand, data storage, video conferencing typically require short connection setup times and even a certain level of protection, which can be specified by a service level agreement between the service provider and the customer. In these agreements, {\cemph connections may be required to be established in a matter of minutes,}%the time limit to setup a connection can be \emph{one minute}, 
%and the time limit to recover from failures can be as low as 50 milliseconds \citep{fawaz2004service,losego2005time}. {\sout{This can be deemed another motivation to tackle the RWA problem in a fast manner. }}

\subsection{Our Work}
\label{subsec:ourwork}
\subsubsection*{Problem definition.}
% OS: what we do
In the light of discussions in the background section, in this study, we consider the RWA problem with static demand and 1:1 dedicated path protection scheme against single-link failures, for a given network with a set of precomputed alternative primary and backup paths and a number of available wavelengths. We call this the {\it Routing and Wavelength Assignment with Protection} {({\PRWA})} problem, where the aim is to primarily  maximize the number of granted requests,
since this brings the actual gains for the service providers \citep{shen2005shared},
while minimizing the wavelength-link usage as a secondary goal to save network resources for future demands. 

\subsubsection*{Outline of our work.} We show that \PRWA~is NP-hard, propose mathematical models and evaluate their solution performance using promising technologies. 
More specifically, we propose an integer programming (IP) model for {\PRWA} as well as a strengthened version of it, and use it as a foundation to develop a novel quadratic unconstrained binary optimization (QUBO) model. 
For model parameters, we present conditions to achieve the intended objective prioritization of granted requests over resource usage, as well as to ensure that the QUBO model is exact.
We use the exact state-of-the-art solver GUROBI to solve the IP models, both by directly providing them to the solver and also applying a problem-specific branch-and-cut method. In order to derive solutions for the QUBO model, we employ a new technology called Digital Annealer (DA), and also test GUROBI. 
%\sout{In order to derive solutions for the QUBO model, we employ a new technology called Digital Annealer, which we compare with three different alternatives using an exact state-of-the-art solver, GUROBI, namely providing the QUBO and IP models directly to the solver, and applying a branch-and-cut method for the latter. }
We conduct computational experiments on a large suite of instances based on commonly used networks, compare the efficiency and efficacy of all our methods as well as a problem-specific heuristic from the literature, and analyze the impact of the penalty coefficient parameter of the QUBO model on DA's performance.

\subsubsection*{Choice of solution technologies.}
%IP has been a widely adopted modeling framework for discrete optimization problems, including RWA problems. This can be attributed to the development of many efficient solution methodologies and enhancements such as the branch-and-cut algorithm, presolve techniques and heuristics; and in turn all the advancements in the IP solvers over the last few decades. 
IP has been a widely adopted modeling framework for discrete optimization problems, including RWA problems. This can be attributed to the development of many efficient solution methodologies and enhancements such as the branch-and-cut algorithm, presolve techniques and heuristics.
Despite their advanced status, state-of-the-art IP solvers may take prohibitively long times to yield optimal solutions, especially for large-scale problems, because run times are typically exponential in the size of the input model. However, for many practical problems, it is sufficient to obtain a good-quality solution, as also noted in \citep{methmod}, but preferably in a reasonable/short amount of time. %, which could be very short in certain cases such as in a matter of minutes. 
%and in certain cases short amount of solution times are also desired. 
Although IP solvers can be tuned to generate feasible solutions faster, they may fail to address both of these practical considerations simultaneously.
%In that regard, the IP approach might still help the former but fail to be satisfactory for the latter. 

% OS: Solvers
%A widely adopted strategy for tackling discrete optimization problems is to formulate them as IP models to be solved by a  state-of-the-art solver. 
%The solution technologies specialized for IP are highly advanced and have long been recognized among the most efficient approaches, notably for RWA problems, too.
%IP technology quite advanced, long been recognized as a viable approach, also notably for RWA problems. 
%State-of-the-art IP solvers are known to perform well particularly for linear models. 
%{\cred However, they may take prohibitively long times to yield optimal or even good-quality solutions especially for large-scale problems, because the run times typically rise exponentially in the size of the input model.}
%For many practical problems, it is sufficient to obtain a good-quality solution, as also noted in \citep{methmod}, and {\cemph  in certain cases short amount of solution times are also desired. }
%previous version: 
%For many practical problems, it is sufficient to obtain a good-quality solution, as also noted in \citep{methmod}, and \sout{usually short amount of solution times are desired; which is exactly the case for our \PRWA~problem. }

%{\cemph In cases where IP tools fail to yield satisfactory performance, }
A plausible alternative to tackle such problems is to formulate them as QUBO models and generate solutions via novel computational architectures and new technologies, such as adiabatic quantum computing (e.g., \citep{papalitsas2019qubo}), neuromorphic computing (e.g., \citep{corder2018solving}), and optical parametric oscillators (e.g., \cite{inagaki2016coherent}), which have recently attracted significant attention due to their capability in tackling combinatorial optimization problems.
A promising example to these new technologies is DA \citep{da}, which is a computer architecture %\citep{aramon2019physics,sao2019application} 
that rivals quantum computers in utility \citep{boyd2018silicon}.
DA is designed to solve QUBO models, and uses an algorithm based on simulated annealing. % \citep{da}. 
In many applications, such as minimum vertex cover problem \citep{javad2019digitally}, maximum clique problem \citep{naghsh2019digitally} and outlier rejection \citep{rahman2019ising}, it has been shown to significantly improve upon the state of the art and yield high-quality solutions in radically short amount of time. 
%RADICALLY SHORT....

\subsubsection*{Contributions.} The contributions of this paper are as follows: 
\BI%[label=--]

\I We prove that {\PRWA} is NP-hard.
\I We develop IP and QUBO models, and propose conditions on parameter values to ensure their validity in the sense that the intended objective prioritization is achieved. Moreover, we propose sufficient conditions establishing the exactness of the QUBO model. %, which has been formally considered by only few studies in the DA literature, e.g., \citep{cohen2020ising, cohen2020constrainedclust}.  
%so that infeasible solutions of the \PRWA~problem are kept inferior to any feasible one in terms of the model objective, and the set of optimal solutions of the problem and the model match and thereby the exactness of the model is established. The latter point has been formally considered by only few studies in the DA literature, e.g., \citep{cohen2020ising, cohen2020constrainedclust}.  
\I Our QUBO approach is the first of its kind in the vast RWA literature and observed to be highly efficient and efficacious {\cemph in solving a large suite of nontrivial test instances}, well addressing the practical needs of \PRWA, which signifies that it can potentially be useful for other important RWA problems due to structural similarities.
%\I We leverage IP tools to solve {\PRWA} ...........................................
%{\cemph \I For cases where IP methods fail to yield an optimal or at least good-quality solution in a reasonable amount of time, we propose a highly efficient and efficacious heuristic solution approach, namely solving the QUBO model with DA. This is first of its kind in the vast RWA literature, signifying that it can potentially be useful for other important RWA problems due to structural similarities. Considering that DA takes at most two minutes in solving a given {\PRWA} instance, it can also aid in dynamic demand cases where it is necessary to obtain solutions quickly. }
\I We show that the emerging DA technology outperforms highly established IP methods accompanied by advanced solvers {\cemph in handling instances that are hard to optimally solve}, indicating that it has the potential to become a viable tool in addressing combinatorial optimization problems. Considering that DA is rarely employed in the operations research literature and only few studies compare it with the state of the art, e.g., \citep{cseker2020digital, naghsh2019digitally, ohzeki2019control,matsubara2020digital}, our study serves as a step to bridge the gap between the use of the established and new promising solution technologies.
%\I We show through sensitivity analysis that models with smaller parameter values lead to significantly better solutions for the considered methods, which has been mentioned in only few studies previously, e.g., \citep{cohen2020ising, cohen2020constrainedclust}. To the best of our knowledge, this is the first reported systematic analysis of the penalty coefficient on DA solution quality.
\I We show through sensitivity analysis that smaller penalty coefficient values lead to better solutions for DA, which has been mentioned in only few studies previously, e.g., \citep{cohen2020ising, cohen2020constrainedclust}. To the best of our knowledge, this is the first reported systematic analysis of the penalty coefficient on DA solution quality.
\EI

\subsubsection*{Paper outline.}
The remainder of this paper is organized as follows.
We review the related literature in Section \ref{sec:litrev}.
In Section \ref{sec:statement}, we present an IP formulation for {\PRWA}, provide a prioritization condition and discuss problem complexity. 
Then, we introduce a QUBO model for {\PRWA} in Section \ref{sec:qubo}, derive a condition that render it exact, and overview the operating principles of DA. 
In Section \ref{sec:comp_study}, we present our computational study, and finally in Section \ref{sec:conclusion}, we conclude our paper with a brief summary.

\section{Related Literature}
\label{sec:litrev}
There are many variants of the RWA problem and a vast number of studies on each type.
%MENTION REVIEW PAPERS HERE MAYBE??
%Many of the RWA problems have been shown to be NP-hard, see for instance \cite{erlebach2001complexity, li2000partition, chlamtac1992lightpath, chiu2000traffic}, and hence they are computationally expensive typically. %, as a result of which most of the existing works concentrate on heuristic solution approaches.
%Therefore, most of the existing studies concentrate on heuristic solution approaches.
%Most of the works concentrate on heuristic solution approaches, and some develop IP formulations as well.
%various solution approaches to solve each, most of them being heuristics.
Here, we restrict our review to the RWA problem with dedicated path protection scheme and static demand, and summarize the most relevant studies in the sequel.
%We begin with two studies that present IP formulations for RWA problems with similar settings to ours.  

\cite{ramamurthy2003survivable} examine different protection approaches for single-link failures, %in wavelength-routed WDM networks, 
and develop IP formulations for path- and link-based schemes, assuming that a precomputed set of alternate routes are given.
% \cite{ramamurthy2003survivable} does not show NP-hardness. They just say that ``The routing and wavelength assignment (RWA) problem (with no protection for any demands) has been shown to be NP-complete [26]. We anticipate that the problems formulated in ILP’s 1–3 are NP-complete as well."
Their model for dedicated path protection aims to minimize the total number of wavelengths used over all links in the network, %considering both primary and backup lightpaths, 
while enforcing that all the demand is satisfied and the wavelength capacity on the links is not exceeded. 
Using test instances generated for a representative network topology, they compare the performance of the proposed IP models. %solved with CPLEX.
Their problem setup differs from ours in that they do not allow unsatisfied demand assuming sufficient capacity in the network, as such use a different objective than we do.
%{\cpink MAYBE WE SHOULD POINT OUT THE FACT THAT THE WAY THEY FORMULATE THE PROBLEM IS DIFFERENT AND THAT WE ARE GOING TO COMPARE OUR MODEL TO THEIRS...}

\cite{azodolmolky2010novel} study the RWA problem with dedicated path protection, where they assume a precomputed set of pairs of primary and backup paths per request, i.e., each primary path has an associated unique backup path. They present two IP models, which form the basis of their heuristic algorithm designed for the impairments aware RWA problem. %, which are enhanced versions of existing ones to incorporate path protection and quality of transmission considerations. % QoT
The first model considers the requests that  require only a primary lightpath, and aims to minimize the total number of requests that are not accepted. The second model extends the first one by (i) adding another demand category that necessitates backup lightpaths as well, and (ii) minimizing a combined sum of the former objective and the maximum number of times a wavelength is used on a link, in a way that acceptance of the requests requiring protection are prioritized over the ones that do not, and the wavelength usage is of least priority. 
%They also propose a heuristic algorithm and evaluate its performance in comparison with the enhanced versions of two heuristics from the literature, one of them being based on the presented IP formulation.  
The problem setting used for their IP models is similar to ours, especially when it is assumed that no unprotected demand exists. However, it does not allow the primary and backup path options of a request to be arbitrarily combined; they are instead considered in predefined pairs only.
{\cemph The authors test the performances of their proposed heuristic algorithm, two versions of the heuristic from the work of \cite{ezzahdi2006lerp} that they enhance with quality of transmission considerations, and solving of the presented IP formulations via an exact solver. 
In our study, we also test the heuristic by \cite{ezzahdi2006lerp} for {\PRWA} and compare it to our presented methods in Section \ref{sec:comp_study}.
}
%{\cred  that they do not actually test the IP performance and simply use heuristic algos....Actually, they do test the performance of solving the ILP formulations with a solver, but do not mention what solver they use etc. The only thing that we could argue is that they only test the direct solving of the ILP formulations, but do not report the results of other IP-specific tools like BC etc and hence do not convince the reader about the inadequacy of IP approach.....}

We note that, in addition to the aforementioned problem setup differences, our study stands apart from those works in terms of modelling. While the previous IP formulations are link-based, our models are path-based. %decision variables {\cred and accordingly constraints in our model are path based.} 
Also, we present an exact QUBO model, the first for an RWA problem although a {\it constrained} binary quadratic model has been used in \citep{ebrahimzadeh2013binary}.

There are some other relevant works that consider the RWA problem with protection for single-link (or node) failures, either with precomputed paths \citep{lee2006comparison}, or by solving both the routing and the wavelength assignment problems simultaneously \citep{937103,1180339} or sequentially \citep{5502801}. 

Lastly, we note that for many variants of the RWA problem, the complexity has been established to be NP-hard, see for instance \citep{erlebach2001complexity, li2000partition, chlamtac1992lightpath, chiu2000traffic}. 
Despite some structural similarities between those variants and our problem, the complexity of {\PRWA} has remained open.

\section{IP Formulation and Problem Complexity}
\label{sec:statement}

\newcommand{\xsol}{\hat{x}}
\newcommand{\ysol}{\hat{y}}
\newcommand{\xsolalt}{\tilde{x}}
\newcommand{\ysolalt}{\tilde{y}}
\newcommand{\xsolaltalt}{\bar{x}}
\newcommand{\ysolaltalt}{\bar{y}}
\newcommand{\objFncVal}{{\obj}(\xsol,\ysol)}
\newcommand{\objFncValalt}{{\obj}(\xsolalt, \ysolalt)}
\newcommand{\objFncValaltalt}{{\obj}(\xsolaltalt, \ysolaltalt)}
\newcommand{\objFncValLink}{{\obj}_\alpha(\xsol,\ysol)}
\newcommand{\objFncValGranted}{{\obj}_\beta(\xsol)}
\newcommand{\objFncValaltLink}{{\obj}_\alpha(\xsolalt, \ysolalt)}
\newcommand{\objFncValaltGranted}{{\obj}_\beta(\xsolalt)}
\newcommand{\objFncValaltaltLink}{{\obj}_\alpha(\xsolaltalt, \ysolaltalt)}
\newcommand{\objFncValaltaltGranted}{{\obj}_\beta(\xsolaltalt)}

\newcommand{\grantedDiff}{k}
\newcommand{\xsolk}{\bar{x}^\grantedDiff}
\newcommand{\ysolk}{\bar{y}^\grantedDiff}
\newcommand{\objValk}{\bar{\obj}^\grantedDiff}
\newcommand{\objValkLink}{\bar{\obj}_{\alpha}^\grantedDiff}
\newcommand{\objValkGranted}{\bar{\obj}_{\beta}^\grantedDiff}
\newcommand{\objValaltalt}{\bar{\obj}}
\newcommand{\objValaltaltLink}{\bar{\obj}_{\alpha}}
\newcommand{\objValaltaltGranted}{\bar{\obj}_{\beta}}

\newcommand{\objVal}{\hat{\obj}}
\newcommand{\objValalt}{\tilde{\obj}}
\newcommand{\objValLink}{\hat{\obj}_{\alpha}}
\newcommand{\objValGranted}{\hat{\obj}_{\beta}}
\newcommand{\objValaltLink}{\tilde{\obj}_{\alpha}}
\newcommand{\objValaltGranted}{\tilde{\obj}_{\beta}}

\newcommand{\ub}{UB}
\newcommand{\maxLPpairLength}{M}

\newcommand{\betaOne}{\beta^1}
\newcommand{\betaExact}{\beta^{\text{Exact}}}

In this section, we first present an IP formulation for the {\PRWA} problem 
%and afterwards prove that it is NP-hard.
in Section \ref{sec:IP-formulation}, followed by Section \ref{subsec:prioritization} where we 
%we explain the constraints and the objective function of our IP formulation, and then 
propose values for the weight parameters used in combining two objectives, namely granted request maximization and link usage minimization, into a single one in order to provably achieve the desired prioritization of the former over the latter.
%maximizing the number of granted requests over minimizing the link usage.
We then show the complexity of {\PRWA} in Section \ref{sec:complexity} through a reduction from a well-known NP-complete problem.

\subsection{IP Formulation}
\label{sec:IP-formulation}

As formally defined in Section \ref{subsec:ourwork}, the {\PRWA} problem aims to grant maximum number of requests by properly assigning a working and a protection lightpath to each from a precomputed collection, while minimizing the wavelength-link usage as a secondary goal, which we hereafter refer to as link usage for simplicity. 

{\cemph
We model an optical network as a directed graph $G=(V,\linkSet)$, with $V$ and $\linkSet$ respectively denoting the set of nodes and the set of directed edges that join ordered pairs of nodes, where it is possible to have multiple arcs with the same start and end nodes. In telecommunications context, we refer the directed edges of the input graph as links.}
 %, where each request is defined as a pair of source and destination nodes. %; that is the demand comes in the form of connection requests between particular nodes. 
We denote the set of requests by $\requestSet$, and the set of wavelengths by $\wavelengthSet$. For each request $\requestInd \in \requestSet$, we represent the set of alternative working and protection lightpaths with $\workingSet^{\requestInd}$ and $\protectionSet^{\requestInd}$, respectively, which are obtained by combining the available precomputed set of paths and wavelengths.
%, where a lightpath is a path in $G$ coupled with a wavelength, and a path is an ordered sequence of links such that the consecutive ones are joined by a common distinct node.
The length of a working (protection) lightpath $ \workingInd $ ($ \protectionInd $) for request $\requestInd \in \requestSet$, i.e., the number of the links it contains, is denoted by $ \length_\workingInd^\requestInd $ ($ \length_\protectionInd^\requestInd $). For convenience, we use $\linksOfLightpath[\ell]$ to represent the set of links that a given lightpath $\ell$ contains, and $\wavelengthOfLightpath[\ell]$ for the wavelength associated with $\ell$. 
%In order to accept a request $\requestInd \in \requestSet$, we need to reserve two lightpaths, one from the working set $ \workingSet^{\requestInd} $, and the other from the protection set $\protectionSet^{\requestInd}$.

%In assigning lightpaths to requests, we need to ensure that the pair of working and protection lightpaths selected to grant a request are link-disjoint,
%Moreover, we should make sure that lightpaths having the same wavelength and sharing a link are not used simultaneously, so that the wavelengths do not interfere with each other.
In order to help compactly represent the constraints of \PRWA, we % incorporate these constraints into our IP model, we 
define four \emph{conflict sets},  $\conflictSet_1, \ldots, \conflictSet_4$.  %, which we explain next.
The first conflict set $\conflictSet_1$ serves to enforce the pair of working and protection lightpaths for a given request to be link-disjoint. 
It is comprised of $(\requestInd, \workingInd, \protectionInd)$ triplets such that the working lightpath $\workingInd$ and the protection lightpath $\protectionInd$ for request $\requestInd$ have at least one link in common.
Namely,
\begin{align*}
    \conflictSet_1  \coloneqq  \left\{ (\requestInd, \workingInd, \protectionInd) \colon \requestInd \in \requestSet, \ \workingInd \in \workingSet^\requestInd, \ \protectionInd \in \protectionSet^\requestInd, \  \linksOfLightpath[\workingInd] \cap \linksOfLightpath[\protectionInd] \neq \varnothing \right\}.
\end{align*}
%\noindent where $\linksOfLightpath[\workingInd] \subseteq E$ and $\linksOfLightpath[\protectionInd] \subseteq E$ represent the set of links that the lightpaths $\workingInd$ and $\protectionInd$ contain, respectively.

The remaining three conflict sets are used to prevent the concurrent use of lightpaths having the same wavelength and sharing a link. 
Considering such lightpaths in pairs, there can be one working and one protection, two working, or two protection lightpaths fulfilling these criteria, which we address through sets $\conflictSet_2$, $\conflictSet_3 $, and $\conflictSet_4$, respectively. 
Let $\conflictSet_2$ be the set of $(\requestInd_1, \requestInd_2, \workingInd, \protectionInd)$ quadruplets such that the working and protection lightpaths $\workingInd$ and  $\protectionInd$ for distinct requests $\requestInd_1$ and $\requestInd_2$ have the same wavelength and at least one link in common:
%{\cred DON"T WE NEED $\requestInd_1 \neq \requestInd_2$??}
%
\begin{align*}
    \conflictSet_2 \coloneqq \left\{ (\requestInd_1, \requestInd_2, \workingInd, \protectionInd) \colon \requestInd_1, \requestInd_2 \in \requestSet, \ \requestInd_1 \neq \requestInd_2, \ \workingInd \in \workingSet^{\requestInd_1}, \ \protectionInd \in \protectionSet^{\requestInd_2}, \ \wavelengthOfLightpath[\workingInd] = \wavelengthOfLightpath[\protectionInd], \ \linksOfLightpath[\workingInd] \cap \linksOfLightpath[\protectionInd] \neq \varnothing \right\}.
\end{align*}
\noindent %where $\wavelengthOfLightpath(\workingInd)$ and $\wavelengthOfLightpath(\protectionInd)$ denote the wavelengths of the working and protection lightpaths $\workingInd$ and $\protectionInd$, respectively. 
The sets $\conflictSet_3$ and $\conflictSet_4$ contain a similar collection of quadruplets as $\conflictSet_2$ does, but with only working and only protection lightpaths, respectively:
\begin{align*}
    \conflictSet_3  &\coloneqq  \left\{ (\requestInd_1, \requestInd_2, \workingInd_1, \workingInd_2) \colon \requestInd_1, \requestInd_2 \in \requestSet, \ \workingInd_1 \in \workingSet^{\requestInd_1}, \ \workingInd_2 \in \workingSet^{\requestInd_2}, \ \wavelengthOfLightpath[\workingInd_1] = \wavelengthOfLightpath[\workingInd_2], \ \linksOfLightpath[\workingInd_1] \cap \linksOfLightpath[\workingInd_2] \neq \varnothing \right\}, \\[0.2cm]
    \conflictSet_4  &\coloneqq  \left\{ (\requestInd_1, \requestInd_2, \protectionInd_1, \protectionInd_2) \colon \requestInd_1, \requestInd_2 \in \requestSet, \ \protectionInd_1 \in \protectionSet^{\requestInd_1}, \ \protectionInd_2 \in \protectionSet^{\requestInd_2}, \ \wavelengthOfLightpath[\protectionInd_1] = \wavelengthOfLightpath[\protectionInd_2], \ \linksOfLightpath[\protectionInd_1] \cap \linksOfLightpath[\protectionInd_2] \neq \varnothing \right\}.
\end{align*}

\begin{example}
\upshape
In Figure \ref{fig:prwa_inst_conflicts}, an example {\PRWA} network is illustrated with two requests together with their working and protection lightpath alternatives. 
The source and destination nodes of the two requests are those with $s^r$ and $t^r$ labels for $\requestInd \in \{1,2\}$, respectively.
The lightpaths are shown with red and green, where each color symbolizes a distinct wavelength, and the lines being solid or dashed indicate whether the lightpath is in the working or protection set, respectively.
The request and working/protection indices of the lightpaths are shown beside them in the same color as the lines representing them.
For {request~1}, there are three working and one protection lightpaths, and for request 2, there is one working and two protection lightpaths.

\begin{figure}[!h]
    \centering
    \tikzset{main_node/.style={circle,draw,line width=1pt,minimum size=0.5cm,inner sep=0pt},}
    \tikzstyle{arc} = [draw,line width=0.7pt,-latex]
    \tikzstyle{edge} = [draw,line width=1.1pt,-]
    \tikzstyle{highlighted red edge} = [draw,line width=1.5pt,-,red, opacity=0.7,rounded corners]
    \tikzstyle{highlighted green edge} = [draw,line width=1.5pt,-,green!70!black, opacity=0.7,rounded corners]

    \scalebox{1}{
    \begin{tikzpicture}
    
    \node[main_node, fill=gray!30!white] (s_a) {$s^1$};

    %\begin{scope}[on background layer]
        
    \node[main_node, fill=gray!80!white] (s_b) [below right = 1.25 cm and 1.5 cm of s_a] {$s^2$};
    \node[main_node] (a1) [above right = 1.25 cm and 1.5 cm of s_a] { $a$};
    \node[main_node, fill=gray!30!white] (t_a) [right = 2 cm of a1] {$t^1$};
    \node[main_node, fill=gray!80!white] (t_b) [below right = 1.25 cm and 1.5 cm of t_a] {$t^2$};
    \node[main_node] (b1) [right = 2 cm of s_b] {$b$ };
    \node[main_node] (mid) [right = 2.625 cm of s_a] { $c$ };

    %\path[edge, every node/.style={sloped,anchor=south,auto=false}]
         %(s_a) to (a1) 
         %(s_a) to (s_b)
         %(s_a) to (mid)
         %(a1) to (mid)
         %(a1) to (t_a)
         %(s_b) to (b1)
         %(s_b) to (mid)
         %(mid) to (t_a)
         %(mid) to (b1)
         %(mid) to (t_b)
         %(b1) to (t_b)
         %(t_a) to (t_b);
         
    \draw[arc] (s_a.60) to (a1.200);
    \draw[arc] (a1.225) to (s_a.35);
    \draw[arc] (s_a.10) to (mid.170);
    \draw[arc] (mid.190) to (s_a.-10);
    \draw[arc] (s_a.-35) to (s_b.135);
    \draw[arc] (s_b.160) to (s_a.-60);
    \draw[arc] (a1.-35) to (mid.115);
    \draw[arc] (mid.140) to (a1.-60);
    \draw[arc] (a1.20) to (t_a.160);
    \draw[arc] (t_a.185) to (a1.-5);
    \draw[arc] (s_b.10) to (b1.170);
    \draw[arc] (b1.195) to (s_b.-15);
    \draw[arc] (s_b.35) to (mid.-115);
    \draw[arc] (mid.-140) to (s_b.60);
    \draw[arc] (mid.70) to (t_a.-140);
    \draw[arc] (t_a.-115) to (mid.45);    
    \draw[arc] (mid.10) to (t_b.170);
    \draw[arc] (t_b.-170) to (mid.-10);    
    \draw[arc] (t_a.-30) to (t_b.130);
    \draw[arc] (t_b.155) to (t_a.-55);   
    \draw[arc] (b1.35) to (t_b.-135);
    \draw[arc] (t_b.-115) to (b1.15);   
    \draw[arc] (mid.-30) to (b1.120);
    \draw[arc] (b1.145) to (mid.-55);   
                     
    %red lightpaths
    \draw[highlighted red edge,-latex,very near start] (s_a.90) -- node[midway, sloped, yshift = .2cm] {\scriptsize $r=1, w=1$} (a1.120) --   (t_a.100); %(r=1,w=1)
    
    \draw[highlighted red edge,-latex,very near start] (s_a.-100) -- ([yshift=-.35cm]s_b.center) --  node[midway, below, yshift = .05cm] {\scriptsize $r=1, w=3$} ([xshift=.1cm,yshift=-.35cm]b1.center) -- ([xshift=.2cm]t_b.0) -- ([yshift=.1cm]t_a.0); %(r=1,w=3)

    \draw[highlighted red edge,-latex,very near start, densely dashed] (s_b.10) --  node[midway, sloped, xshift=0.15cm, yshift = -.15cm] {\scriptsize $r=2, p=1$} ([yshift=-.3cm]mid.center) -- (t_b.-120) ; %(r=2,p=1)
    
    \draw[highlighted red edge,-latex,very near start] (s_b.-50) -- ([xshift=0.1cm, yshift=-0.2cm]b1.center) --  node[midway,sloped, xshift=-0.05cm, yshift=.3cm] {\scriptsize $r=2, w=1$} (t_b.-95); %(r=2,w=1)
    
    %green lightpaths
    \draw[highlighted green edge,-latex,very near start] (s_a.15) -- ([xshift=.05cm, yshift=-.2cm]a1.45) -- ([yshift=.25cm]mid.center) -- node[midway, sloped, yshift=0.2cm] {\scriptsize $\ r=1, w=2$} (t_a.195); %(r=1,w=2)
    
    \draw[highlighted green edge,-latex,very near start, densely dashed] (s_a.35) -- node[midway, xshift = 0cm, yshift = .2cm] {\scriptsize $r=1, p=1$} ([xshift=.3cm, yshift=.15cm]mid.center) --  (t_a.-90); %(r=1,p=1)

    \draw[highlighted green edge,-latex,very near start, densely dashed] (s_b.110) -- ([xshift=0.1cm]s_a.-30) -- node[midway, yshift=-0.2cm] {\scriptsize $r=2, p=2$} ([yshift=-.15cm]mid.center) -- (t_b.210); %(r=2,p=2)
    
    %\end{scope}
  
    \end{tikzpicture}
    }
    \caption{An {\PRWA} network and two requests with their precomputed working and protection lightpaths.} 
    \label{fig:prwa_inst_conflicts}
\end{figure}
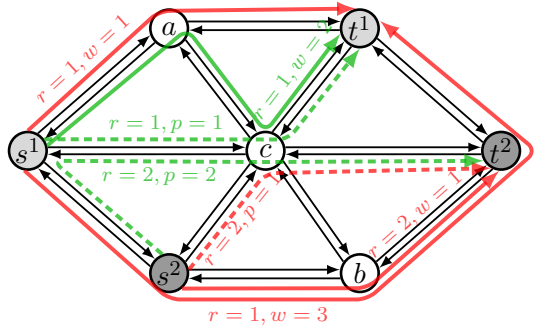

Let us give some example tuples for the conflict sets using the network in Figure \ref{fig:prwa_inst_conflicts}.
The link $(c, t^1)$ is common in the lightpaths labeled with $\requestInd=1, \workingInd=2$ and $\requestInd=1, \protectionInd=1$, which yields $(\requestInd, \workingInd, \protectionInd) = (1,2,1) \in \conflictSet_1$.
%and that it is shared by a pair of working and protection paths having the same wavelength (green) gives  $(\requestInd_1, \requestInd_2, \workingInd, \protectionInd) = (1,1,2,1) \in \conflictSet_2$. %{\cred THE SAME AS THE PREVIOUS ONE??}
Furthermore, the link $(s^2, b)$ being contained in two lightpaths having the same wavelength (red) makes $(\requestInd_1, \requestInd_2, \workingInd_1, \workingInd_2) = (1,2,3,1) \in \conflictSet_3$, and $(s^1, c)$ being shared by two protection lightpaths with the same wavelength (green) leads to $(\requestInd_1, \requestInd_2, \protectionInd_1, \protectionInd_2) = (1,2,1,2) \in \conflictSet_4$.
Since there is no pair of distinct requests whose working and protection lightpaths have the same wavelength, $\conflictSet_2 = \varnothing$ here. 

In this example, it is possible to accept both of the requests by selecting the working and protection lightpaths $\workingInd = 1$ and $\protectionInd = 1$ for request $\requestInd = 1$, and also for $\requestInd = 2$. This solution is indeed the best option for link usage as well, because having granted all of the given requests with lightpaths of length two, it is not possible to use any fewer links as each lightpath is of length at least two in this example.
\hfill %$\Halmos$
\end{example}

Using the notation introduced above and  two sets of binary decision variables defined as	
\begin{align*}
  x_\workingInd^\requestInd \ &= \ 
  \begin{cases}
    1, \quad &\text{if working lightpath }  \workingInd \in \workingSet^\requestInd  \text{ is assigned to request } \requestInd \in \requestSet\\
    0, \quad &\text{otherwise} 
  \end{cases}
\\
  y_\protectionInd^\requestInd \ &= \ 
  \begin{cases}
    1, \quad &\text{if protection lightpath }  \protectionInd \in \protectionSet^\requestInd \text{ is assigned to request } \requestInd  \in \requestSet\\
    0, \quad &\text{otherwise}
  \end{cases}
\end{align*}%
\noindent we now present a novel IP formulation as follows:  %
\begin{subequations}
\label{m:IP}
\begin{alignat}{3}
\min \quad & \objFnc \ \coloneqq \ \alpha \sum_{\requestInd\in \requestSet} \left( \sum_{\workingInd\in \workingSet^\requestInd} \length_\workingInd^\requestInd x_\workingInd^\requestInd \ + \ \sum_{\protectionInd \in \protectionSet^\requestInd} \length_\protectionInd^\requestInd y^\requestInd_{\protectionInd}  \right)  
\ - \ 
\beta \ \sum_{\requestInd\in \requestSet}\sum_{\workingInd \in \workingSet^\requestInd} x_\workingInd^\requestInd \span \span \label{eq:OF} \\[0.25cm] %use one \span command for each & you want to ignore
\text{s.t.} \quad & \sum_{\workingInd \in \workingSet^\requestInd} x_\workingInd^\requestInd \ - \ \sum_{\protectionInd \in \protectionSet^\requestInd} y_\protectionInd^\requestInd \ = \ 0 \qquad \qquad && \requestInd \in \requestSet \label{eq:w_p_equality} \\[0.1cm]
&\sum_{\workingInd \in \workingSet^\requestInd} x_\workingInd^\requestInd  \ \leq \ 1 \qquad \qquad &&\requestInd \in \requestSet  \label{eq:w_p_at_most_one}\\[0.1cm]
& x_{\workingInd}^{\requestInd} \ + \ y_{\protectionInd}^{\requestInd} \ \leq \ 1 \qquad \qquad && (\requestInd, \workingInd,\protectionInd) \in \conflictSet_1 \label{eq:conflicts1}\\[0.1cm]
& x_{\workingInd}^{\requestInd_1} \ + \ y_{\protectionInd}^{\requestInd_2} \ \leq \ 1 \qquad \qquad &&  (\requestInd_1, \requestInd_2, \workingInd, \protectionInd) \in \conflictSet_{2} \label{eq:conflicts2}\\[0.1cm]
& x_{\workingInd_1}^{\requestInd_1} \ + \ x_{\workingInd_2}^{\requestInd_2} \ \leq \ 1 \qquad \qquad &&  (\requestInd_1, \requestInd_2, \workingInd_1, \workingInd_2) \in \conflictSet_{3}  \label{eq:conflicts3}\\[0.1cm]
& y_{\protectionInd_1}^{\requestInd_1} \ + \ y_{\protectionInd_2}^{\requestInd_2} \ \leq \ 1 \qquad \qquad &&  (\requestInd_1, \requestInd_2, \protectionInd_1, \protectionInd_2) \in \conflictSet_{4}  \label{eq:conflicts4}\\[0.1cm]
& x_\workingInd^\requestInd, \ y_{\protectionInd}^\requestInd \ \in \ \{0,1\} \qquad  \qquad && \requestInd \in \requestSet, \ \workingInd \in \workingSet^\requestInd, \ \protectionInd \in \protectionSet^\requestInd \label{eq:domain}
\end{alignat}%
\end{subequations}%
\noindent where $\alpha$ and $\beta$ are predetermined positive constants.

Constraint set \eqref{eq:w_p_equality} enforces that  the same number of working and protection lightpaths are selected to grant a request, and \eqref{eq:w_p_at_most_one} ensures that at most one working lightpath is assigned to each request.   
Constraint set \eqref{eq:conflicts1} guarantees that the selected working and protection lightpaths for each request are link-disjoint, while \eqref{eq:conflicts2}--\eqref{eq:conflicts4} make sure that the lightpaths having the same wavelength and sharing a link are not chosen simultaneously.
Finally, constraint set \eqref{eq:domain} states the domains of the decision variables. 
%In the rest of the paper, we refer to a point $(x,y)$ satisfying the domain constraints \eqref{eq:domain} as a \emph{binary solution}, and call it \emph{feasible/infeasible to the IP} if it satisfies/violates all/any of  \eqref{eq:w_p_equality}-\eqref{eq:conflicts4}.

\begin{comment}
%We decided not to do this. Their formulation may well be stronger than ours.
{\cpink 
\begin{claim}
Our formulation is stronger than the ones in \cite{ramamurthy2003survivable,azodolmolky2010novel} ??? 
\end{claim}
}
\end{comment}

The objective function \eqref{eq:OF} combines the two goals of {\PRWA}, minimizing the number of links used and maximizing the number of requests granted, as a weighted sum.
%of  and , where the latter term has a minus sign as we want to maximize it.
%In reality, it is the granting of received requests that achieves the actual gains for the service providers \citep{shen2005shared}; yet, saving from link usage can potentially leave room in the network for future demands. Hence, the primary goal is specified as the maximization of request granting, while the secondary one is meant for searching alternative allocation of routes with possibly fewer link usage.
%{\cpink LINK THE SECONDARY OBJECTIVE TO WAVELENGTH LINK USAGE WHICH IS USED IN THE LITERATURE, FOR INSTANCE IN \cite{zang2003path}!!!}
As mentioned in the introduction, the latter goal must be prioritized over the former, which we detail next.

\subsection{Objective Prioritization}
\label{subsec:prioritization}
We now formally define what prioritization of request granting over link usage means, and propose $\alpha$ and $\beta$ values that serve the purpose in \eqref{eq:OF}.
We first introduce some notation to be used in the sequel.
Let $\objFncLink$ and $\objFncGranted$ be two functions respectively corresponding to the number of links used and the number of requests granted at a solution, i.e.,
\begin{align*}
\objFncLink  \ \coloneqq \ \sum_{\requestInd\in \requestSet} \left( \sum_{\workingInd\in \workingSet^\requestInd} \length^\requestInd_\workingInd \ x_\workingInd^\requestInd \ + \ \sum_{\protectionInd \in \protectionSet^\requestInd} \length^\requestInd_\protectionInd \ y^\requestInd_{\protectionInd}  \right), \quad 
%\\[0.15cm]
% 
\objFncGranted \ \coloneqq \ \sum_{\requestInd\in \requestSet}\sum_{\workingInd \in \workingSet^\requestInd} x_\workingInd^\requestInd, 
\end{align*}
so that the objective function \eqref{eq:OF} can be equivalently written as
\begin{align*}
\objFnc \  = \ \alpha \ \objFncLink \ - \ \beta \ \objFncGranted.    
\end{align*}

\noindent 

Furthermore, to ease the presentation, for any given feasible solution  $(\xgiven, \ygiven)$ (where {\tiny $^\bullet$} represents any operator such as hat, tilde, and bar), %$\hat{\phantom{x}}, \tilde{\phantom{x}}, \bar{\phantom{x}}$),  %and $(\xsolalt, \ysolalt)$ be two solutions feasible to the IP model \eqref{m:IP}, for which 
we respectively define the associated IP objective value and its components as
$$\fgiven \ \coloneqq \ \obj(\xgiven,\ygiven), \quad
\fgiven_{\alpha} \ \coloneqq \  \obj_{\alpha}(\xgiven,\ygiven), \quad 
\fgiven_{\beta} \ \coloneqq \ \obj_{\beta}(\xgiven),$$
and the worst-case and best-case link usage of a feasible solution granting the same number of requests as 
\begin{align*}
&\fgiven_{\alpha}^{\text{max}} \  \coloneqq \max \left\{ \objFncLink ~|~ (x,y) \text{ satisfies } \eqref{eq:w_p_equality}-\eqref{eq:domain}, \objFncGranted = \fgiven_{\beta} \right\}, \\ % \limits_{\objFncGranted = \objValGranted}\{\objFncLink\} 
& \fgiven_{\alpha}^{\text{min}} \  \coloneqq \min \left \{\objFncLink ~|~ (x,y) \text{ satisfies } \eqref{eq:w_p_equality}-\eqref{eq:domain}, \objFncGranted = \fgiven_{\beta} \right\}. % \limits_{\objFncGranted = \objValaltGranted}\{\objFncLink\}
\end{align*}

\begin{definition}[Prioritization Condition]
\label{def:prioritization}
Request granting is prioritized over link usage, if for any pair of feasible solutions $(\xsol, \ysol)$ and $(\xsolalt, \ysolalt)$ with $ \objValGranted > \objValaltGranted $, we have $\objVal < \objValalt$, i.e., the marginal contribution of granting a request to the objective function \eqref{eq:OF} is always negative. That is,   
$
    \alpha \objValLink - \beta \objValGranted   <  \alpha \objValaltLink - \beta \objValaltGranted   \  \label{eq:prioritization_v0}
$
for all  feasible $(\xsol, \ysol)$ and $(\xsolalt, \ysolalt)$  with $\objValGranted > \objValaltGranted$.
\hfill %\Halmos
\end{definition}

Considering the largest and smallest realizations of the left- and right-hand sides in terms of link usage, respectively, the prioritization condition can be equivalently written as
\begin{align}
    \alpha \objValLink^{\text{max}} - \beta \objValGranted  \ < \ \alpha \objValaltLink^{\text{min}} - \beta \objValaltGranted  \ \  
    \text{for all feasible } (\xsol, \ysol) \text{ and } (\xsolalt, \ysolalt) \text{ with } \objValGranted > \objValaltGranted.
  \label{eq:prioritization}
\end{align}

\noindent This condition can also be expressed with the help of an optimization model: 
\begin{align}
\frac{\beta}{\alpha} 
\ > \ 
\LBprioExact
\ \coloneqq \
\max \left\{\frac{ \objValLink^{\text{max}} - \objValaltLink^{\text{min}}  }{\objValGranted - \objValaltGranted} 
\colon
(\xsol, \ysol) \text{ and } (\xsolalt, \ysolalt) \text{ are feasible with } \objValGranted > \objValaltGranted
\right\}. \label{eq:prioritization_optmodel0}
\end{align}

\noindent Indeed, it is possible define another optimization model by only considering the solutions differing by one in their number of granted requests,
\begin{align}
\LBprioOne
\ \coloneqq \
\max \left\{\frac{ \objValLink^{\text{max}} - \objValaltLink^{\text{min}}  }{\objValGranted - \objValaltGranted} 
\colon
(\xsol, \ysol) \text{ and } (\xsolalt, \ysolalt) \text{ are feasible with } \objValGranted = \objValaltGranted + 1
\right\},  \label{eq:prioritization_optmodel1}
\end{align}
which would achieve what \eqref{eq:prioritization_optmodel0} does, as provided in the  proposition below.

\begin{proposition}%[{\cpink ??????????????????}]
\label{prop:beta_model_onediff}
The optimization models in \eqref{eq:prioritization_optmodel0} and \eqref{eq:prioritization_optmodel1} yield the same optimal values; that is, $\LBprioExact = \LBprioOne$. This implies that the prioritization condition given in \eqref{eq:prioritization_optmodel0} can also be achieved by setting $\alpha, \beta > 0$ such that 
$\frac{\beta}{\alpha}  >  \LBprioOne $.

\end{proposition}
\begin{proof}
See \ref{OS:proof:prop_beta_one_model}. \hfill %\Halmos
\end{proof}

{\cemph
For practical purposes, we assume $\alpha = 1$ and that $\beta$ can only take integer values.
In this case, letting $ \betaTight$ denote the smallest integer $\beta$ value satisfying the prioritization condition in \eqref{eq:prioritization_optmodel1}, we have  $ \betaTight = 1 + \LBprioOne $.
However, obtaining $\betaTight$ or at least a reasonable upper bound for it necessitates solving of a non-trivial optimization problem, which could be computationally even more expensive than the original {\PRWA} problem.
% (probably even more challenging than the original one), which requires some computation time and thus does not actually comply with our goal to solve the {\PRWA} problem quickly.
%However, in cases where solution technologies like DA need the magnitudes of parameter values not to exceed certain limits, one may solve a preliminary optimization model to find a valid $\betaTight$ value and thereby make the instance solvable with the associated solver.
Next, we derive sufficient condition for prioritization in terms of given instance parameters. %, without any need to solve an optimization problem.
%In the light of the above results, we next derive a sufficient condition for prioritization in terms of given instance parameters. 
}

%{\cpink In order to prioritize request granting, its weight $ \beta $ should be selected sufficiently larger than the weight $ \alpha $ of link usage.}
%Proposition \ref{prop:obj_weights} provides a condition for $\alpha$ and $\beta$ that achieves the intended prioritization of request granting. %, by establishing that the marginal contribution of accepting a request to the objective function \eqref{eq:OF} is always negative.

\begin{proposition}[IP objective weight selection]
\label{prop:obj_weights}
%\begin{proposition}[IP objective weight selection]
%\label{prop:obj_weights}
Selecting $\alpha, \beta > 0$ such that 
\begin{align}
    \frac{\beta}{\alpha} \ &> \ |\requestSet| \ (\maxLPpairLength - 2) \ + \ 2 \label{eq:obj_coef_ratio_suff_lb}
\end{align}
prioritizes request granting over link usage in \eqref{eq:OF} for any feasible solution to the IP, i.e., solutions accepting more requests yield lower objective values, where ${\maxLPpairLength = \max\limits_{\requestInd \in \requestSet}\left\{ \max\limits_{\workingInd \in \workingSet^\requestInd} \big\{ \length^{\requestInd}_{\workingInd} \big\} + \max\limits_{\protectionInd \in \protectionSet^\requestInd}\big\{ \length^{\requestInd}_{\protectionInd} \big\} \right\}}$.
%where $\maxLPpairLength = \max_{\requestInd \in \requestSet}\left\{ \length^{\requestInd}_{wmax} +  \length^{\requestInd}_{pmax} \right\}$ with $ \length^{\requestInd}_{wmax} = \max_{\workingInd \in \workingSet^\requestInd}\left\{ \length^{\requestInd}_{\workingInd} \right\}$ and  $ \length^{\requestInd}_{pmax} = \max_{\protectionInd \in \protectionSet^\requestInd}\left\{ \length^{\requestInd}_{\protectionInd} \right\}$. 
%\end{proposition}
%

\end{proposition}
\begin{proof}
See  \ref{OS:proof:prop_obj_weights}. \hfill %\Halmos
\end{proof}

Proposition \ref{prop:obj_weights} provides a lower bound on $\frac{\beta}{\alpha}$ that is sufficient to make request granting the primary goal. 
Note that computing this lower bound does not involve solution of an optimization problem and hence makes it easy to decide on safe objective parameter combinations for a given instance. %, thus serves well to our fundamental aim of solving the \PRWA~problem quickly. 
{\cemph Letting $\betaBase$ denote the least possible $\beta$ value from the condition in \eqref{eq:obj_coef_ratio_suff_lb} (assuming $\alpha = 1$), we have $ {\betaBase \ = \ |\requestSet| \ (\maxLPpairLength - 2) \ + \ 3}$.
Note that by definition $\betaBase \geq \betaTight $.
}

%In Proposition \ref{prop:obj_weights_tightLB}
We now show that there exist examples where the bound in \eqref{eq:obj_coef_ratio_suff_lb} is indeed tight.
%; namely, a value for $\frac{\beta}{\alpha}$ that does not exceed the right-hand side of \eqref{eq:obj_coef_ratio_suff_lb} is not enough to prioritize request granting over link usage. 

\begin{proposition}[Tight example for the weight selection]
\label{prop:obj_weights_tightLB}
% \begin{proposition}[Tight example for the weight selection]
% \label{prop:obj_weights_tightLB}
There exist {\PRWA} instances for which the lower bound provided in Proposition \ref{prop:obj_weights} is %tight 
necessary to prioritize request granting over link usage.
%for binary solutions.
% \end{proposition}

\end{proposition}
\begin{proof}
See \ref{OS:proof:prop_obj_weights_tightLB}. \hfill %\Halmos
\end{proof}

\newcommand{\conflictSetAug}{\conflictSet^{\text{Aug}}}
\newcommand{\protectionSubset}{\tilde{\protectionSet}}
\newcommand{\workingSubset}{\tilde{\workingSet}}
\newcommand{\protectionSubseti}{\bar{\protectionSet}}
\newcommand{\protectionIndi}{\protectionInd^{\prime}}
\newcommand{\workingIndi}{\workingInd^{\prime}}

%{\cemph

\subsection{Strengthened Conflict Constraints}
%DEFINE $\protectionSubset^{\requestInd, \workingInd}$ HERE, AND THEN CLEAN UP THE FOLLOWING CONFLICT SET..........SIMILARLY FOR THE OTHER CONFLICT SET DEFS.......ACTUALLY, ONLY DEFINE SUCH SETS AND DO NOT DEFINE NEW CONFLICT TUPLES; SIMPLY WRITE THE CONSTRAINTS FOR ALL REQ OR FOR ALL e, lambda....

We can strengthen the set of conflict constraints in  \eqref{eq:conflicts1}--\eqref{eq:conflicts4} by identifying larger groups of variables that are mutually in conflict, i.e., groups in which at most one variable can take value one.
To this end, we first construct the set $\protectionSubseti^{\requestInd, \workingInd}$ of all protection lightpaths for request $\requestInd$ that has at least one common link with working lightpath $\workingInd$, as an extension of the lightpath pairs defined in $\conflictSet_1$. That is, for every $\workingInd \in \workingSet^{\requestInd}$ and request $ \requestInd \in \requestSet$, we define
\begin{align*}
   \protectionSubseti^{\requestInd, \workingInd} \coloneqq \big\{\protectionInd \in \protectionSet^\requestInd \colon  \linksOfLightpath[\workingInd] \cap \linksOfLightpath[\protectionInd] \neq \varnothing \big\}.
\end{align*}
Next, we define the sets of working and protection lightpaths that contain link $\link$ and wavelength $\wavelength$, for every wavelength $\wavelength \in \wavelengthSet$, link $\link \in \linkSet$, and request $ \requestInd \in \requestSet$, which will extend the lightpath pairs defined in conflict sets  $\conflictSet_2, \conflictSet_3$ and $\conflictSet_4$. 
Namely,
\begin{align*}
   \workingSubset^{\requestInd}_{\link, \wavelength} \coloneqq \big\{ \workingInd \in \workingSet^\requestInd \colon \wavelengthOfLightpath[\workingInd] = \wavelength, \ \link \in \linksOfLightpath[\workingInd] \big\},
\end{align*}
and
\begin{align*}
   \protectionSubset^{\requestInd}_{\link, \wavelength} \coloneqq \big\{\protectionInd \in \protectionSet^\requestInd \colon \wavelengthOfLightpath[\protectionInd] = \wavelength, \ \link \in \linksOfLightpath[\protectionInd] \big\}.
\end{align*}
Using these sets, we can write a strengthened form of our conflict constraints as
\begin{subequations}
\label{m:IP_strong}
\begin{alignat}{3}
& x_{\workingInd}^{\requestInd} \ + \ \sum_{\protectionInd \in \protectionSubseti^{\requestInd, \workingInd}} y_{\protectionInd}^{\requestInd} \ \leq \ 1 \qquad \qquad &&  \requestInd \in \requestSet, \  \workingInd \in \workingSet \label{eq:conflicts_strong1}\\[0.1cm]
& \sum_{\requestInd \in \requestSet} \sum_{\workingInd \in \workingSubset^{\requestInd}_{\link, \wavelength}} x_{\workingInd}^{\requestInd} \ + \ \sum_{\requestInd \in \requestSet} \sum_{\protectionInd \in \protectionSubset^{\requestInd}_{\link, \wavelength} } y_{\protectionInd}^{\requestInd} \ \leq \ 1 \qquad \qquad && \link \in \linkSet, \ \wavelength \in \wavelengthSet. \label{eq:conflicts_strong2} %\\[0.1cm]
%
%& \sum_{\requestInd \in \requestSet} \sum_{\workingInd \in \workingSet^{\requestInd}_{\link, \wavelength}} x_{\workingInd}^{\requestInd} \ \leq \ 1 \qquad \qquad && \link \in \linkSet, \ \wavelength \in \wavelengthSet \label{eq:conflicts_strong3}\\[0.1cm]
%
%& \sum_{\requestInd \in \requestSet} \sum_{\protectionInd \in \protectionSet^{\requestInd}_{\link, \wavelength}} y_{\protectionInd}^{\requestInd} \ \leq \ 1 \qquad \qquad && \link \in \linkSet, \ \wavelength \in \wavelengthSet \label{eq:conflicts_strong4}
%
\end{alignat}
\end{subequations}

In the rest of the paper, we refer to the IP model in \eqref{m:IP} as {\IPbase}, and to that with the conflict constraints in \eqref{eq:conflicts1}--\eqref{eq:conflicts4} being replaced with the ones in \eqref{eq:conflicts_strong1}--\eqref{eq:conflicts_strong2} as {\IPstrong}.

%}

\subsection{Complexity}\label{sec:complexity}
In this section, we prove that the {\PRWA} problem is NP-hard. More specifically, we prove that a special case of our problem is NP-hard by making a reduction from the {\it maximum stable set} problem, which is known to be NP-hard \citep{garey1979computers}. 

Given a graph $ G_s \left( V_s, E_s \right) $, a {\it stable set} is a set $V_s^\prime \subseteq V_s$ of nodes such that no two nodes in $V_s^\prime$ are linked by an edge in $E_s$.
The goal in the maximum stable set problem ({\mss}) is to find a stable set of maximum cardinality in $G_s$.  
The decision version of the problem, which we denote by {\dmss}, is concerned with the existence of a stable set of size at least $k$ in the input graph $G_s$, for some given integer $k \geq 1$.

Now, we define {\PRWAreq} as the special version of the {\PRWA} problem that aims to maximize the number of granted requests only (thus the extension ``{\sc -r}"), i.e., the variant which can be modeled as \eqref{m:IP} using $\alpha = 0$ and $\beta = 1$ in the objective function \eqref{eq:OF}. Let {\DPRWAreq} be the decision version of {\PRWAreq}, which checks whether at least $k$ requests can be granted for some given integer $k \geq 1$.   
In what follows, we show that  {\PRWAreq} is NP-hard, starting with
the polynomial-time verifiability for its decision version.

\begin{lemma}[Verifiability]
\label{lem:np}
% \begin{lemma}[Verifiability]
% \label{lem:np}
{\DPRWAreq} is in NP.
% \end{lemma}

\end{lemma}
\begin{proof}
See \ref{OS:proof:lem_np}. \hfill %\Halmos
\end{proof}

%Next, we prove the hardness of the {\PRWAreq} problem.

\begin{theorem}[Complexity]
\label{thm:npcomplete}
% \begin{theorem}[Complexity]
% \label{thm:npcomplete}
	{\PRWAreq} is NP-hard.
% \end{theorem}

\end{theorem}
\begin{proof}
See  \ref{OS:proof:thm_npcomplete}. \hfill %\Halmos
\end{proof}

As %mentioned before, {\PRWAreq} is equivalent to {\PRWA} if we set the objective coefficient parameters $\alpha = 0$ and $\beta = 1$, which makes 
{\PRWA} generalizes {\PRWAreq}, it is at least as hard, which yields the desired result.
%Therefore, we have the following corollary to Theorem \ref{thm:npcomplete}.

\begin{corollary}
\label{cor:nphard} 
{\PRWA} is NP-hard.
\end{corollary}

\section{QUBO Formulation and Solution Method}
\label{sec:qubo}

%SAY THAT IT IS ALSO POSSIBLE TO EMBED QUADRATIC OBJECTIVE AS WELL...
In this section, we present our proposed modeling and solution approach for the {\PRWA} problem.
In Section \ref{subsec:qubo}, we present a QUBO %Quadratic Unconstrained Binary Optimization (QUBO) 
model via a transformation from {\IPbase} obtained by dualizing its constraints, %of the IP, namely \eqref{eq:w_p_equality}--\eqref{eq:conflicts4}, 
i.e., adding them as a penalty term to the objective function. %in \eqref{eq:OF}.
%Our proposed QUBO model is exact, i.e., an optimal solution to the QUBO model also constitutes an optimal solution to the IP formulation, thus to the {\PRWA} problem.
We also explain how to carefully choose the penalty parameters
%%In Section \ref{subsec:qubo}, we first explain how to construct the penalty terms that would be positive valued at binary solutions infeasible to the IP, and then show how to carefully choose the penalty coefficients 
to achieve the exactness of the QUBO model.
Then, in Section \ref{subsec:da}, we overview the Digital Annealer technology and its operating principles.

\subsection{Transformation to QUBO}
\label{subsec:qubo}
As the first step of obtaining an exact QUBO formulation, we dualize the constraints of our IP formulation given in \eqref{m:IP}, {\IPbase}, in such a way that any infeasible solution to it, i.e., any constraint violation, yields a strictly positive penalty term in the objective function of our QUBO model.
This is easy to achieve for equality constraints; any linear equality constraint can be transformed into a penalty term by simply taking the square of the difference of its left and right-hand sides, so that any constraint violation translates into a positive penalty value, and hence can be avoided through the minimization of the objective.
Therefore, for our only set of equality constraints \eqref{eq:w_p_equality}, the corresponding penalty term includes for each $\requestInd \in \requestSet$ the following squared violation expression:
$$ \left(\sum_{\workingInd \in \workingSet^\requestInd} x_\workingInd^\requestInd \ - \ \sum_{\protectionInd \in \protectionSet^\requestInd} y_\protectionInd^\requestInd \right)^2,$$ which amounts to a positive value when more working lightpaths than protection lightpaths are selected for the request $r$, or vice versa.

In case of inequality constraints, however, more custom-tailored approaches are needed, because violations occur in one direction only. In order to transform the inequality constraints in \eqref{eq:w_p_at_most_one}--\eqref{eq:conflicts4} into penalty terms, we first reformulate them as quadratic \emph{equality} constraints in \eqref{eqs:equalities}. 
\begin{subequations}
\label{eqs:equalities}
\begin{alignat}{2}
& \left(\sum_{\workingInd \in \workingSet^\requestInd} x_\workingInd^\requestInd \right) \left( \sum_{\workingInd \in \workingSet^\requestInd} x_\workingInd^\requestInd  - 1 \right)  \ = \ 0 \quad \qquad &&\requestInd \in \requestSet \label{eq:w_p_at_most_one_quad}\\[0.1cm]
& x_{\workingInd}^{\requestInd} \ y_{\protectionInd}^{\requestInd} \ = \ 0 \qquad \qquad && (\requestInd, \workingInd,\protectionInd) \in \conflictSet_1 \label{eq:conflicts1_quad}\\[0.1cm]
& x_{\workingInd}^{\requestInd_1} \ y_{\protectionInd}^{\requestInd_2} \ = \ 0 \qquad \qquad &&  (\requestInd_1, \requestInd_2, \workingInd, \protectionInd) \in \conflictSet_{2} \label{eq:conflicts2_quad}\\[0.1cm]
& x_{\workingInd_1}^{\requestInd_1} \ x_{\workingInd_2}^{\requestInd_2} \ = \ 0 \qquad \qquad &&  (\requestInd_1, \requestInd_2, \workingInd_1, \workingInd_2) \in \conflictSet_{3}  \label{eq:conflicts3_quad}\\[0.1cm]
& y_{\protectionInd_1}^{\requestInd_1} \ y_{\protectionInd_2}^{\requestInd_2} \ = \ 0 \qquad \qquad &&  (\requestInd_1, \requestInd_2, \protectionInd_1, \protectionInd_2) \in \conflictSet_{4}  \label{eq:conflicts4_quad}
\end{alignat}
\end{subequations}

Constraint set \eqref{eq:w_p_at_most_one_quad} is the quadratic equivalent of \eqref{eq:w_p_at_most_one} ensuring that at most one lightpath is selected per request. As the decision variables are binary, the left-hand side of \eqref{eq:w_p_at_most_one}, i.e., the expression denoting the total number of working lightpaths assigned to request $r$, can take value either zero or one, in which case the left-hand side of \eqref{eq:w_p_at_most_one_quad} becomes zero. So, \eqref{eq:w_p_at_most_one_quad} holds only when the corresponding original constraint  \eqref{eq:w_p_at_most_one} is satisfied, otherwise, i.e., when $\sum_{w\in W^\requestInd} x_w^\requestInd \geq 2$, the left-hand side of \eqref{eq:w_p_at_most_one_quad} takes a strictly positive value.
Therefore, the left-hand side of \eqref{eq:w_p_at_most_one_quad} can be used as a penalty term for violations of constraints \eqref{eq:w_p_at_most_one}.
Similarly, constraints \eqref{eq:conflicts1}--\eqref{eq:conflicts4}, each of which ensuring that the two involved lightpaths cannot be both selected due to a conflict, are violated only when both variables on the left-hand side take value one; all other configurations of the two binary variables are feasible.
The quadratic constraints \eqref{eq:conflicts1_quad}--\eqref{eq:conflicts4_quad} take advantage of the fact that all feasible configurations involve at least one variable having value zero, and force the product of the two to be zero. 
So, when the associated constraints are violated, the left-hand sides of \eqref{eq:conflicts1_quad}--\eqref{eq:conflicts4_quad} take strictly positive values, namely value one, thus serve as penalty terms to be added to the objective function of our QUBO model.
Note that the magnitude of violation that an infeasible binary solution creates in any one of the constraints \eqref{eq:w_p_equality}--\eqref{eq:conflicts4} is at least one, which has a useful role in rendering our QUBO formulation exact, as we will see when we specify possible values of the penalty coefficient in the sequel.

We present our QUBO formulation for {\PRWA} in \eqref{m:QUBO}.
\begin{subequations}
\label{m:QUBO}
\begin{alignat}{2}
\min \qquad & \alpha \ \sum_{\requestInd\in \requestSet} \left( \sum_{\workingInd \in \workingSet^\requestInd} \length_\workingInd^\requestInd x_\workingInd^{\requestInd} + \sum_{\protectionInd \in \protectionSet^\requestInd} \length_\protectionInd^\requestInd y_\protectionInd^{\requestInd}  \right) \ - \ \beta \ \left(\sum_{\requestInd\in \requestSet}\sum_{\workingInd \in \workingSet^\requestInd} x_\workingInd^\requestInd\right) \nonumber \\[0.1cm]
& + \ \penaltyCoef \sum_{\requestInd\in \requestSet}\left(\sum_{\workingInd \in \workingSet^\requestInd} x_\workingInd^\requestInd - \sum_{\protectionInd \in \protectionSet^\requestInd} y_\protectionInd^\requestInd \right)^2 \ + \ \penaltyCoef \ \sum_{\requestInd\in \requestSet}\left(\sum_{\workingInd \in \workingSet^\requestInd} x_\workingInd^\requestInd -1 \right) \left(\sum_{\workingInd \in \workingSet^\requestInd} x_\workingInd^\requestInd \right)  \nonumber \\[0.15cm]
& + \  \penaltyCoef  \sum_{(\requestInd, \workingInd,\protectionInd) \in \conflictSet_1} x_{\workingInd}^{\requestInd} \ y_{\protectionInd}^{\requestInd} \ \ \ + \ \ \ \penaltyCoef  \sum_{(\requestInd_1, \requestInd_2, \workingInd, \protectionInd) \in \conflictSet_{2}} x_{\workingInd}^{\requestInd_1} \ y_{\protectionInd}^{\requestInd_2}  \nonumber \\[0.15cm]
& + \ \penaltyCoef \sum_{(\requestInd_1, \requestInd_2, \workingInd_1, \workingInd_2) \in \conflictSet_{3}} x_{\workingInd_1}^{\requestInd_1} \ x_{\workingInd_2}^{\requestInd_2} \ \ \ + \ \ \ \penaltyCoef \sum_{(\requestInd_1, \requestInd_2, \protectionInd_1, \protectionInd_2) \in \conflictSet_{4} } y_{\protectionInd_1}^{\requestInd_1} \ y_{\protectionInd_2}^{\requestInd_2}   \label{eq:OF_qubo} \\[0.25cm]
\text{s.t.} \qquad & x_\workingInd^\requestInd, \ y_{\protectionInd}^\requestInd \in \{0,1\}  \qquad \qquad \requestInd \in \requestSet, \ \workingInd \in \workingSet^\requestInd, \ \protectionInd \in \protectionSet^\requestInd , \label{eq:domain_qubo}
\end{alignat}
\end{subequations}
\noindent where $ \penaltyCoef > 0$ is the penalty coefficient for the dualized constraints. We note that different penalty coefficients can be used for different terms, however, we choose them to be all the same, $\penaltyCoef$, to simplify our derivation of a valid lower bound for it.

{\cemph
We note that {\IPstrong} can also be transformed into a QUBO model, in which case the strengthened set of conflict constraints in \eqref{eq:conflicts_strong1}--\eqref{eq:conflicts_strong2} would simply be dualized in the same manner constraints in \eqref{eq:w_p_at_most_one} are dualized (see quadratic equalities in \eqref{eq:w_p_at_most_one_quad}).
This yields sums of bilinear penalty terms in the objective, similar to those corresponding to the conflict constraints in \eqref{eq:conflicts1}--\eqref{eq:conflicts4}.
Thus, the two QUBO models can be made equivalent by customizing the penalty coefficient values for the associated terms.

Next, we investigate conditions ensuring that an optimal solution to the QUBO model is also optimal for the {\PRWA} problem.
}
%{\cemph ADD A BRIEF DISCUSSION ON HOW THE TWO QUBO MODELS WOULD BE EQUIVALENT IF WE WERE TO CUSTOMIZE RHO COEFS FOR EACH TERM SEPARATELY!!!!}

\begin{definition}[Exactness]
\newcommand{\model}{\mathcal{M}}
\newcommand{\prob}{\mathcal{P}}
A model $\model$ for a problem $\prob$ is \emph{exact} if any optimal solution to $\model$ is feasible and optimal for $\prob$.
\end{definition}

By construction, {\IPbase} (and also {\IPstrong}) is an exact model for the {\PRWA} problem.
On the other hand, for the QUBO formulation to be exact, the penalty coefficient $ \penaltyCoef $ should be selected ``sufficiently large".
Since high valued parameters might lead to serious numerical issues, smaller ``safe" values are desirable. In that regard, we provide a lower bound for $\penaltyCoef$ that is sufficient to guarantee that the QUBO model is exact.  %To the best of our knowledge, such bounds have not been often derived in the literature. 

\begin{proposition}[QUBO penalty selection]
\label{prop:penalty_coef}
% \begin{proposition}[QUBO penalty selection]
% \label{prop:penalty_coef}
When 
\begin{align}
\penaltyCoef 
\ > \
\beta  (|\requestSet| \ + \ 1) 
\ - \ 
\alpha \left( 1 \ + \  \sum_{\requestInd \in \requestSet} \left(\length^{\requestInd}_{w_{\min}}  + \ \length^{\requestInd}_{p_{\min}} \right) \right),  \label{eq:penalty_coef_lb}  
\end{align}
\eqref{m:QUBO} is an exact QUBO model for the {\PRWA} problem,
%; that is, any optimal solution of the QUBO model is feasible and optimal for the IP model.
where 
$\length^{\requestInd}_{w_{\min}} 
= \ 
\min_{\workingInd \in \workingSet^{\requestInd}} \left \{\length^\requestInd_\workingInd \right \} $
and 
$\length^{\requestInd}_{p_{\min}}
= \ 
\min_{\protectionInd \in \protectionSet^{\requestInd}} \left \{\length^\requestInd_\protectionInd \right \}$.
% \end{proposition}

\end{proposition}
\begin{proof}
See  \ref{OS:proof:prop_penalty_coef}. \hfill %\Halmos
\end{proof}

It is important to note that the condition in \eqref{eq:penalty_coef_lb} not only guarantees the exactness of the QUBO model, but also ensures that any infeasible solution for the problem is inferior to the feasible ones. 
We chose to impose this stronger requirement in deriving the lower bound on the penalty coefficient in order to establish a clear dominance relationship between the classes of feasible and infeasible solutions, which we believe leads to a conceptually better QUBO model. 
We also note that the resulting lower bound is not tight, as far as the original definition of exactness is concerned. If $\rho^{\text{Base}}$ is a penalty coefficient abiding \eqref{eq:penalty_coef_lb}, then similar to the objective weight parameter discussion in the IP case, we can actually design an optimization model to obtain the smallest possible penalty coefficient value, $\rho^{\text{Tight}}$. However, the resulting model would be much more complex (e.g., a 0-1 quadratic fractional programming model).

%Proposition \ref{prop:penalty_coef} provides a condition for the penalty coefficient $\penaltyCoef$ that is sufficient to render the QUBO formulation in \eqref{m:QUBO} exact, i.e., any $\penaltyCoef$ value strictly above the suggested lower bound is guaranteed to deliver exactness.

%Next, we show that there exist {\PRWA} instances for which this sufficient condition is necessary, too. %; that is, there are cases where the suggested lower bound is tight.

%\input{QUBO_subfiles/prop_penalty_tight_ex}

%{\cred [Oylum: TODO] Comment: Instance based, can be further improved. }
%We note that while there exist instances for which the condition provided in \eqref{eq:penalty_coef_lb} needs to be necessarily satisfied as shown in Proposition \ref{prop:penalty_coef_tightLB}, for other cases, the suggested lower bound can be improved by making use of the specifics of a particular {\PRWA} instance or a class of instances, which may help improve the performance of IP and QUBO solvers.
%{\cpink what else to say about this???}

Given a QUBO model, we can optimally solve it using the state-of-the-art solvers like GUROBI \citep{gurobi} and CPLEX \citep{cplex}. 
However, if the model is originally constrained and linear, as it is in our case, a more favorable approach would be to use these solvers to solve the IP formulations, in which they are particularly successful.
The main limitation of the IP solvers is that their performance deteriorates as the number of variables and constraints increases, with an exponential rise in solution times typically, as a result of which they likely fail to deliver an optimal or good-quality solution for realistic problem sizes in a short amount of time.
For problems suitable to be formulated as a QUBO model, a promising alternative to the state of the art mentioned above is the {\it Digital Annealer} technology, which in theory is not affected by the increasing number of variables and constraints, and demonstrates a robust level of performance across instances having different sizes, as long as the number of variables does not exceed the allowed variable capacity. 
Next, we provide some information on this technology.

\begin{comment}
what to add to this section:

WE HAVE A QUBO AT HAND, WE CAN SOLVE THIS USING CPLEX GUROBI ETC. AS A PROMISING ALTERNATIVE, WE USE DA BLAH BLAH

IN THEORY, DA WONT BE AFFECTED BY INCREASING NUMBER OF CONSTRAINTS... EMPHASIZE THIS ADVANTAGE!!! AND JUSTIFY THAT WE THEREFORE PREFER IT AND IT IS SUITABLE FOR PRWA
MENTION THE DISADVANTAGES OF CPLEX GUROBI THAT THEY CANNOT SCALE WHILE DA CAN

SIMULATED ANNEALING'IN NASIL CALISTIGINI DA BIR IKI CUMLE ILE ANLAT
\end{comment}

\subsection{The Digital Annealer}
\label{subsec:da}

The Digital Annealer (DA) is a quantum-inspired computer architecture designed to derive solutions for combinatorial optimization problems formulated as a QUBO model. % \cite{da}.
It consolidates the merits of both quantum and general-purpose computers, and takes advantage of the massive parallelization that its hardware allows  \citep{da, sao2019application}. 
The first generation of DA is capable of solving problems with up to 1024 variables, while this number has increased to 8192 in the second generation \citep{daWeb}.%, which is the one that we utilize in this study. 

The algorithm of DA is based on simulated annealing.
Simulated annealing (SA) is a probabilistic method for finding solutions to combinatorial optimization problems that aim to minimize some cost function, by making an analogy to the physical process of {\it annealing} whereby a heated material is slowly cooled until it reaches a state of minimum energy \citep{kirkpatrick1983optimization, bertsimas1993simulated}.
%, called the ground state . 
%on an analogy between the minimization of a cost function of a combinatorial optimization problem and the slow cooling of a material until it reaches the minimum energy state (ground state) \cite{kirkpatrick1983optimization}.
%Starting from an initial solution, 
%Four main ingredients are needed for SA: (1) a concise description of the {\it state} of a system, which corresponds to a representation of a {\it solution} to the problem, (2) definition of a {\it move} operator to make a transition from one state to the other, (3) an objective function, and (4) a cooling schedule for temperatures \cite{kirkpatrick1983optimization}.
%
%Four main ingredients are needed to apply SA: (1) a solution representation (state) (2) a {\it move} operator to define a one-step transition from one solution to the other, (3) an objective function, and (4) a cooling schedule for temperatures \cite{kirkpatrick1983optimization}.
The idea in simulated annealing is to propose a random perturbation to the current solution at each iteration, evaluate the consequent change in the objective function, and decide whether or not to move to the proposed solution. 
If the proposed solution results in a lower objective value, it is always accepted; otherwise, i.e., if it is a ``uphill" move, it is accepted with a probability that is a function of the change in the objective value and the current temperature. %\citep{kirkpatrick1983optimization, rutenbar1989simulated}.
%, with higher temperatures and smaller increases in the objective making the acceptance more likely.
While higher temperatures more likely permit uphill moves to let the algorithm explore a larger region of the objective function and to help escape from local optima, the search intensifies around a narrower area with lower temperatures.  
Under certain conditions, simulated annealing asymptotically converges to a global optimum, yet, it may necessitate infinitely many iterations.
So, in practice, it is very well possible to converge to a local optimum in simulated annealing \citep{kirkpatrick1983optimization, rutenbar1989simulated, glover2006handbook, gendreau2010handbook}.

To apply simulated annealing based algorithms, one needs to define a {\it solution representation} as well as a {\it move} operation to propose a new candidate solution at each iteration \citep{kirkpatrick1983optimization}.
In DA, a solution (to a QUBO problem) is represented with a vector of binary variable values, and the move operation is defined as the {\it flip} of a variable value, i.e., changing the value of a variable from one to zero or vice versa.

While being grounded in simulated annealing, DA's algorithm differs from it in some key aspects. % \citep{da}.
First, it uses a {\it parallel trial} scheme, where it evaluates all possible moves in parallel at each iteration, as opposed to the classical way of considering one random move only.  
When more than one flip is eligible for acceptance, one of them is chosen uniformly at random.
Second, it utilizes a {\it dynamic offset} mechanism to escape from local optima, such that if no flip is accepted in the current iteration, the acceptance probabilities in the subsequent iteration are artificially increased. % \cite{da}.
{\cemph Specifically, when no candidate variable to flip can be found, a positive offset value is added to the objective function, equivalent to multiplying the acceptance probabilities with a coefficient that is a function of the current temperature and the magnitude of the offset.  Otherwise, the offset value is set to zero.}
Third, DA has the {\it parallel tempering} option, also referred to as the {\it replica exchange} method, where multiple independent search processes (replicas) are initiated in parallel with a different temperature each, and states (solutions) are probabilistically exchanged between them.
This way, each replica performs a random walk in the temperature space, helping to avoid being stuck at a local minimum \citep{da, hukushima1996exchange, matsubara2020digital}.
{\cemph
In our computational experiments, we utilize DA in parallel tempering mode.
}

\section{Computational Study}
\label{sec:comp_study}

In this section, we present the results of our computational study.
We generated a large suite of \PRWA~instances using known networks from the literature and conducted detailed analysis.
Our experimental setting can be summarized as follows:
\BI%[label={--}]
\I {\bf Solvers.} We used the second generation of DA\footnote[1]{For DA experiments, we used the Digital Annealer environment prepared exclusively for research at the University of Toronto.}  
\citep{matsubara2020digital} and GUROBI 9.0\footnote[2]{For GUROBI experiments, we used a MacOS computer with 3 GHz Intel Core i5 CPU and 16 GB memory.} \citep{gurobi}.
\I {\bf Methods.} While we (1) provided the QUBO formulation to DA, we (2) employed GUROBI in three different ways; 
(i) to directly solve the IP formulations, both {\IPbase} and {\IPstrong},  
(ii) to solve {\IPstrong} via branch-and-cut (B\&C) using the lazy callback feature, and 
(iii) to directly solve the QUBO formulation. 
We note that we sometimes refer to (i) as ``GUROBI as IP solver", and to (iii) as ``GUROBI as QUBO solver". 
In addition, we solve {\PRWA} via the random-search-based heuristic from \citep{ezzahdi2006lerp}, which was also utilized in \citep{azodolmolky2010novel}.
We refer to this heuristic as {\RSheur}.
%We employ DA in parallel tempering mode in our experiments. 
\I {\bf Time limit.} We used three different time limits; 
%60 seconds with reference to the services where fast response times are desired \citep{fawaz2004service, losego2005time}, 
120 seconds, which is approximately the highest run time DA takes for our particular problem,
 %to make an additional performance comparison when a longer run time is allowed for both solvers, 
 600 seconds to compare the longer run-time performance of GUROBI to that of DA, and
 7200 seconds (two hours) to see how well GUROBI can achieve in cases where hours of run times are tolerable.
 %For {\RSheur}, we experimented with different numbers of iterations, corresponding to a range from 2 to 900 seconds.
\I {\bf Experiments.}  We carried out three main groups of analyses; 
(1) performance comparison of the five alternative {\methods}, % using the considered time limits, 
(2) effect of using solutions from DA as an initial solution for GUROBI, as well as a run time analysis for GUROBI to reach DA's performance level, and
(3) sensitivity analysis of DA's performance to values of the penalty coefficient $\penaltyCoef$.
%(2) comparison of DA results to the optimal (or best-known) values, 
%(3) a run time analysis for GUROBI to reach DA's performance level, and 
%(4) sensitivity analysis of model parameters using $\betaTight$ values and various quantities for the penalty coefficient $\penaltyCoef$.
\I {\bf Implementation details.}
First, we implemented a B\&C algorithm using the callback features of GUROBI to pass the conflict constraints as needed. We tried various cut selection and management strategies for user and lazy callbacks, the best of which yielded no better performance than using the lazy callbacks, i.e. by adding a conflict constraint each time a feasible solution violating it is encountered, which, therefore, is the one we utilize in our experiments.
%First, we tested two alternative ways of implementing a B\&C algorithm in GUROBI; (i) by using the callback features to pass the conflict constraints to GUROBI ourselves as needed, and (ii)  by providing all the conflict constraints to GUROBI as lazy constraints when the model is initialized. We tried various cut selection and management strategies for user and lazy callbacks for option (i), the best of which yielded basically the same level of performance with option (ii). As such, we continued our experiments with the second alternative.
Second, although it is not possible to explicitly impose a time limit for DA, after some preliminary testing, we set the number of iterations to an appropriate value yielding the desired execution time of 120 seconds, which indeed was the highest time we observed for our particular problem, {\PRWA}.
Once the number of iterations are fixed, DA's execution times show almost no variability across different instances.
%Therefore, we determined two values for the number of iterations and fixed them for our 60- and 120-second experiments in DA, while keeping the values of other run parameters intact.
Third, we failed to utilize the penalty coefficient values suggested in Proposition \ref{prop:penalty_coef} because the values were outside the acceptable range for DA. 
Thus, we used smaller values for the penalty coefficient that do not necessarily guarantee the exactness of the model, but always yielded feasible solutions in practice. 
\EI

The remainder of this section is organized as follows.
In Section \ref{subsec:instances}, we provide some details about the networks used and the way we generated our problem instances.
In Section \ref{subsec:performances}, we present our main set of experimental results, follow it by an initial solution and run time analysis in Section \ref{subsec:runtime_analysis}, and then provide the results of our penalty coefficient analysis in Section \ref{subsec:pentrials}.
Finally in Section \ref{subsec:observations}, we summarize our key observations as a result of our computational study.

\subsection{Problem Instances}
\label{subsec:instances}

In order to generate our test instances, we made use of four different network topologies of varying sizes and densities from the literature, namely EON \citep{tornatore2007wdm}, Brazil \citep{jaumard2006much}, USA and China \citep{hwang2009one}.
%Atlanta \citep{orlowski2010sndlib}, NSF \citep{ramaswami1996design}, and COST239 \citep{tan1996wavelength}, 
%Selected characteristics of these networks, parameter values used to generate our instances from these networks, and instance size-related 
In Table \ref{tab:network_info}, we provide descriptive information about the networks and the test instances we generated from them.

\begin{table}[!h]
	\centering
	\caption{Instance information.}
	\label{tab:network_info}%
	\scalebox{0.7}{
    \hspace*{-0.5cm}
    		\begin{tabular}{c S[table-format=2.0] S[table-format=3.0] S[table-format=1.2] cc  c S[table-format=1.2] S[table-format=3.1]}
			\toprule
			& \multicolumn{3}{c}{Network features} & \multicolumn{2}{c}{Instance parameters}  & \multicolumn{3}{c}{Instance sizes}  \\
			\cmidrule(lr){2-4} \cmidrule(lr){5-6} \cmidrule(lr){7-9}
			  \parbox{1.9cm}{\centering Network} &   
			  \parbox{1.5cm}{\centering \# nodes ($|V|$)} &  
			  \parbox{1.25cm}{\centering \# links ($|\linkSet|$)} & 
			  \parbox{2.5cm}{\centering Avg \\ in/out-degree } & 
			  \parbox{2.5cm}{\centering \# wavelengths ($|\wavelengthSet|$)} & 
			  \parbox{2.4cm}{\centering  \# requests  ($|\requestSet|$)} &
			  \parbox{2.25cm}{\centering  \# vars } &
			  \parbox{2.8cm}{\centering  \# cons / \# vars in {\IPstrong}} &
			  \parbox{2.8cm}{\centering   \# cons / \# vars in {\IPbase}}
			  \\
			\midrule
			EON   &  19  &  78  & 2.05  & \{5, 10, 15\} & \{60, 80, 100\}       & [2400, 8000] & 0.67  & 234.0   \\
			USA   &  24   &  86   & 1.79 & \{5, 10, 15\}  & \{60, 80, 100\}      & [2400, 8000] & 0.69  & 212.0   \\
			Brazil   &  27  &  140  &  2.67  & \{5, 10, 15\}  & \{60, 80, 100\}   & [2400, 8000] & 0.77  & 131.3  \\
			China   &  39  &  144  &  1.85  & \{5, 10, 15\}  & \{60, 80, 100\}   & [2400, 8000] & 0.79  & 174.5  \\
			\bottomrule 
		\end{tabular}%	      
	} 
\end{table}%	      

\begin{comment}
\begin{table}[!h]
	\centering
	\caption{Instance information.}
	\label{tab:network_info}%
	\scalebox{0.75}{
    \def\arraystretch{1.25}%  1 is the default
    \setlength\aboverulesep{1.5pt} %0pt default
    \setlength\belowrulesep{2.5pt} % 0.65ex default  
    		\begin{tabular}{c S[table-format=2.0] S[table-format=3.0] S[table-format=1.2] ccc}
			\toprule
			& \multicolumn{3}{c}{Network features} & & \multicolumn{2}{c}{Instance generation parameters} \\
			\cmidrule(lr){2-4} \cmidrule(lr){6-7}
			  \parbox{2cm}{\centering Network} &   
			  \parbox{2.25cm}{\centering \# nodes ($|V|$)} &  
			  \parbox{2cm}{\centering \# links ($|\linkSet|$)} & 
			  \parbox{3cm}{\centering Avg in/out-degree } & & 
			  \parbox{3.25cm}{\centering \# wavelengths ($|\wavelengthSet|$)} & 
			  \parbox{2.75cm}{\centering  \# requests  ($|\requestSet|$)} \\
			\midrule
			EON   &  19  &  78  & 2.05  && \{5, 10, 15\} & \{60, 80, 100\} \\
			USA   &  24   &  86   & 1.79 && \{5, 10, 15\}  & \{60, 80, 100\} \\
			Brazil   &  27  &  140  &  2.67  && \{5, 10, 15\}  & \{60, 80, 100\} \\
			China   &  39  &  144  &  1.85  && \{5, 10, 15\}  & \{60, 80, 100\} \\
			\bottomrule
		\end{tabular}%
	}
\end{table}%
\end{comment}

For all the networks, we assume that each edge $\{u,v\}$ is represented with a pair of links $(u,v)$ and $(v,u)$.
%the links (the edges of the networks) are bidirectional, i.e., links can be used for information transmission in both directions (from either end of the edge).
We use three different wavelength capacities on links, $|\wavelengthSet| \in \{5,10,15\}$, and three different numbers for requests, $|\requestSet| \in \{60,80,100\}$.
We generate five random instances for each parameter combination, which makes a total of $180$ test instances. 
%created $ 100 $ random instances per network, making a total of 300 test instances.
For each instance, we selected a distinct source and destination pair for every request among all possible ordered node pairs in the network.
For each request,  we randomly selected four working and four protection paths between the source and destination nodes, except when $|\wavelengthSet| = 15$ and $|\requestSet| \in \{80,100\}$ where we decreased the number of working/protection paths to three and two, so that the sizes of those instances become eligible for DA, which can handle at most 8192 variables.
We formed the working and protection lightpaths of every request by combining each generated path with every one of the available wavelengths.

As indicators of instance sizes, we provide the number of variables (``\# vars") as well as the ratio of the number of constraints to the number of variables (``\# cons / \# vars"). both for {\IPstrong} and {\IPbase} formulations. It is noteworthy that the strengthened set of our conflict constraints leads to a remarkable decrease in the total number of constraints; the ratio of the number of constraints to the number of variables is 170 to 350 times higher in {\IPbase} than in {\IPstrong}.
%{\cemph Our test set consists of nontrivial instances, for which it is typically hard to obtain optimal solutions.}
 %and hence in the ratio of the number of constraints to the number of variables, which is 170 to 350 times higher in {\IPbase} than in {\IPstrong}.
%The maximum number of variables that the second generation of DA can handle being 8192, our instances are eligible to be tested on it. 

%{\cpink JUSTIFY THE PARAMETERS USED IN INSTANCE GENERATION, USING THEIR RATIOS AND COMPARING THEM TO REALISTIC ONES FOR INSTANCE!!! 
%Note by Oylum: I will do this later, when I am finished with the rest of the article...}

\subsection{Performance Comparison of the \Methods}
\label{subsec:performances}

In this section, we report the results of our main set of experiments and compare the performances of all the {\methods} under consideration, including {\RSheur}.  
For our particular problem {\PRWA}, the steps of {\RSheur} are as follows: 
At each step, we generate a random permutation of the request set.
%Initially, we generate a set of random permutations of the request set. 
Considering the requests in the order they appear in the permutation, we first find a link-disjoint pair of working and protection paths, starting the search with the shortest paths in the precomputed set of paths for the associated request. If such a pair exists, we search for the first available wavelength whose assignment does not lead to a conflict with the set of lightpaths assigned to the current set of granted requests, if any. If no wavelength is available, we select the next pair of working and protection paths from the precomputed set (following an increasing order in path lengths), and seek the first available wavelength as before. 
This procedure continues until we find a feasible pair of lightpaths or no alternative is left to consider. %Each request in the sequence is handled in the same way.
In our implementation, we continue considering different permutations of the set of requests until the imposed time limit is reached, and a 120-second time limit yielded an average of 7243 permutations per instance in our test set.
At the end, we select a solution with maximum number of granted requests over the set of all considered permutations.

Table \ref{tab:performances_summary} reports the number of granted requests for all {\methods} averaged over all wavelength and request numbers, and the average number of links used per lightpath of a granted request in parentheses in a second row for each network, under two groups of columns corresponding to the experiments with 120- and 600-second time limits. %, and with {\RSheur}. 
For direct solving of  {\IPstrong} with GUROBI, we present additional sets of results obtained by setting the solver's emphasis on search for feasible solutions (``{\IPstrong} (f)") for both time limits, while all other GUROBI-related results are obtained under default settings.
%It also contains the range of the number of variables in our instances (``\# vars"), as well as the average ratio of the number of constraints to the number of variables (``${\# \text{~cons}~/~\# \text{~vars}}$") both for {\IPstrong} and {\IPbase}.

% Table generated by Excel2LaTeX from sheet 'SUMMARY'
\begin{table}[htbp]
  \centering
  \caption{Summary of overall results in terms of the number of granted requests (average number of links per granted request shown in parentheses).}
   \label{tab:performances_summary}%
   \hspace*{-0.25cm}
   \scalebox{0.73}{
    \begin{tabular}{cccccccccc}
    \toprule
          &  \multicolumn{7}{c}{120 sec}   &  \multicolumn{2}{c}{600 sec} \\
    	 \cmidrule(lr){2-8} \cmidrule(lr){9-10}
	  &   &  & \multicolumn{5}{c}{GUROBI}  &   \multicolumn{2}{c}{GUROBI}  \\
	 \cmidrule(lr){4-8} \cmidrule(lr){9-10}
          \parbox{2.25cm}{\centering Network} 
	  & \parbox{1.75cm}{\centering DA} & \parbox{1.9cm}{\centering {\RSheur}} & \parbox{1.9cm}{\centering  \IPstrong~(f)} & \parbox{1.9cm}{\centering \IPstrong} & \parbox{1.6cm}{\centering \IPbase} & \parbox{1.6cm}{\centering B\&C} & \parbox{1.6cm}{\centering QUBO} 
	  &  \parbox{2cm}{\centering  \IPstrong~(f)} & \parbox{2cm}{\centering \IPstrong} 
	  \\
    \midrule
    EON    &   23.3    &   13.7    &   21.0   & 20.6     & 16.7   & 2.2        & 6.1      &  21.6   & 21.0  	 \\ 
    	        &   (9.1)   &   (9.4)    &   (9.0)  & (9.2)    & (9.3)   & (9.0)     & (8.8)    &   (9.0)  & (9.0)  	 \vspace*{0.1cm} \\ 
    USA    &   24.7     &   15.3    &   22.6   & 22.2    & 19.9    & 2.3       & 5.0        &  23.2   & 22.7 	\\
    	        &   (8.9)   &   (9.0)    &   (8.9)   & (9.1)    & (8.9)   & (8.5)    & (8.9)    &  (8.9)  & (8.9)	 \vspace*{0.1cm} \\ 
    Brazil   &  43.7     &   28.1    &   37.7   & 38.1    & 30.8    & 7.7       & 7.6       &  41.0   & 40.0	 \\
    	        &   (8.1)   &   (8.0)    &   (8.1)   & (8.2)    & (8.1)   & (7.8)     & (7.5)   &  (8.1)   & (8.1)	\vspace*{0.1cm} \\ 
    China  &   31.6    &   19.4      &   28.4   & 28.2    & 22.9    & 2.1       & 5.1        &  29.2   & 28.7 \\
    	        &  (10.4)  &  (10.2)    &  (10.3)  & (10.5) & (10.3) & (10.1)   & (9.7)    &  (10.3) & (10.4)  \\ 
    \bottomrule
    \end{tabular}%
    }
\end{table}%

We observe from Table \ref{tab:performances_summary} that the number of granted requests obtained from DA in 120 seconds outperform those obtained from 120- and even 600-second experiments with GUROBI for all networks. 
Among the GUROBI-based alternatives, on the other hand, solving of {\IPstrong} directly with GUROBI yields the best results, most of the time slightly better when the emphasis is on finding feasible solutions.
Even though the results from {\RSheur} are better than those obtained from solving {\IPstrong} through B\&C (with callback) and the QUBO formulation directly with GUROBI, they are all still far from being comparable to the other alternatives.
%{\cemph MOVE THE RELEVANT INFO TO INSTANCE INFO TABLE AND THE COMMENTS TO THEREAFTER... The strengthened set of our conflict constraints leads to a remarkable decrease in the total number of constraints, and hence in the ratio of the number of constraints to the number of variables, which is 170 to 350 times higher in {\IPbase} than in {\IPstrong}.}
Also, {\IPstrong} yields considerably better results than {\IPbase}, as expected, and increasing the time limit from 120 to 600 seconds leads to a relative improvement of only 0.4 to 1.9 in the average number of granted requests for {\IPstrong}.
%For the 60-second results, the average number of granted requests from DA is significantly better than that from {\IPstrong} for the instances based on the Brazil network, and almost the same for the remaining ones. Also, it is noteworthy that increasing the time limit from 120 to 600 seconds leads to a relative improvement of only 0.4 to 1.9 in the average number of granted requests for GUROBI.
As for the link usage, since it is our secondary objective, two solutions with different numbers of granted requests cannot be compared on the basis of the total number of links used.
Nevertheless, in terms of the average number of links used per lightpath of a granted request, DA in 120 seconds and GUROBI in 600 seconds perform almost always the same, even though DA achieves a higher number of granted requests.
%Nevertheless, we note that even though DA achieves a higher number of granted requests in 120 seconds, the average number of links used per granted request is the same as in 600-second {\IPstrong} solutions.

One might expect the B\&C method to compete with feeding the IP formulation to the solver as a whole; however, B\&C performs even worse than directly solving the QUBO formulation with GUROBI mostly.
In order to see whether it could provide good-quality solutions when longer run times are allowed, we also tested it with a 600-second time limit on a selected set of 30 instances based on USA and Brazil networks, with $(|\wavelengthSet|, |\requestSet|) \in \{(5,60), (10,80), (15,100)\}$, where $|\wavelengthSet|$ denotes the number of wavelengths and $|\requestSet|$ the number of requests.
For each pair of parameter values, we have ten instances in this selected set, five per network.
For half of these instances based on USA network, the average number of granted requests rises from 3.3 to 7.5, and for the remaining ones based on Brazil network it rises from 4.8 to 10.2.
Consequently, we conclude that the B\&C method fails to be a favorable {\method} for further consideration, as do solving {\IPbase} and QUBO formulations directly with GUROBI. 
We also tested the heuristic {\RSheur} under a 600-second time limit, and observed that the increase in the average number of granted requests is less than one, confirming that this method too is not a favorable option for further consideration.
%In addition, since the 120-second performance of solving {\IPbase} and QUBO formulations directly with GUROBI were outperformed by that of directly solving {\IPstrong} with GUROBI, we did not conduct 600-second experiments for them.

We confine the rest of our analysis to the most promising two alternatives; namely, solving {\IPstrong} using GUROBI with an emphasis on finding feasible solutions, and solving the QUBO formulation with DA.
Figure \ref{fig:DA_vs_IP} compares the average number of granted requests for every $(|\wavelengthSet|, |\requestSet|)$ pair, using 120- and 600-second time limits for DA and GUROBI, respectively. %of the number of wavelengths ($|\wavelengthSet|$) and number of requests  ($|\requestSet|$)
The four plots demonstrate that DA outperforms GUROBI for each parameter combination and every one of the four networks, and the difference is particularly evident when the number of wavelengths is ten.
In fact, DA not only outperforms in terms of the average values but also provides predominantly superior or otherwise as good results in almost all individual instances.
These results show that {\cemph for our test set comprising instances that are hard to optimally solve}, DA yields solutions within only two minutes which are better than or as good as the ones GUROBI attains in ten minutes. %over all considered networks and instance generation parameters.

\begin{figure}[!h]
    \centering
    \begin{subfigure}[t]{0.35\textwidth}
	\centering
        \includegraphics[scale=0.4]{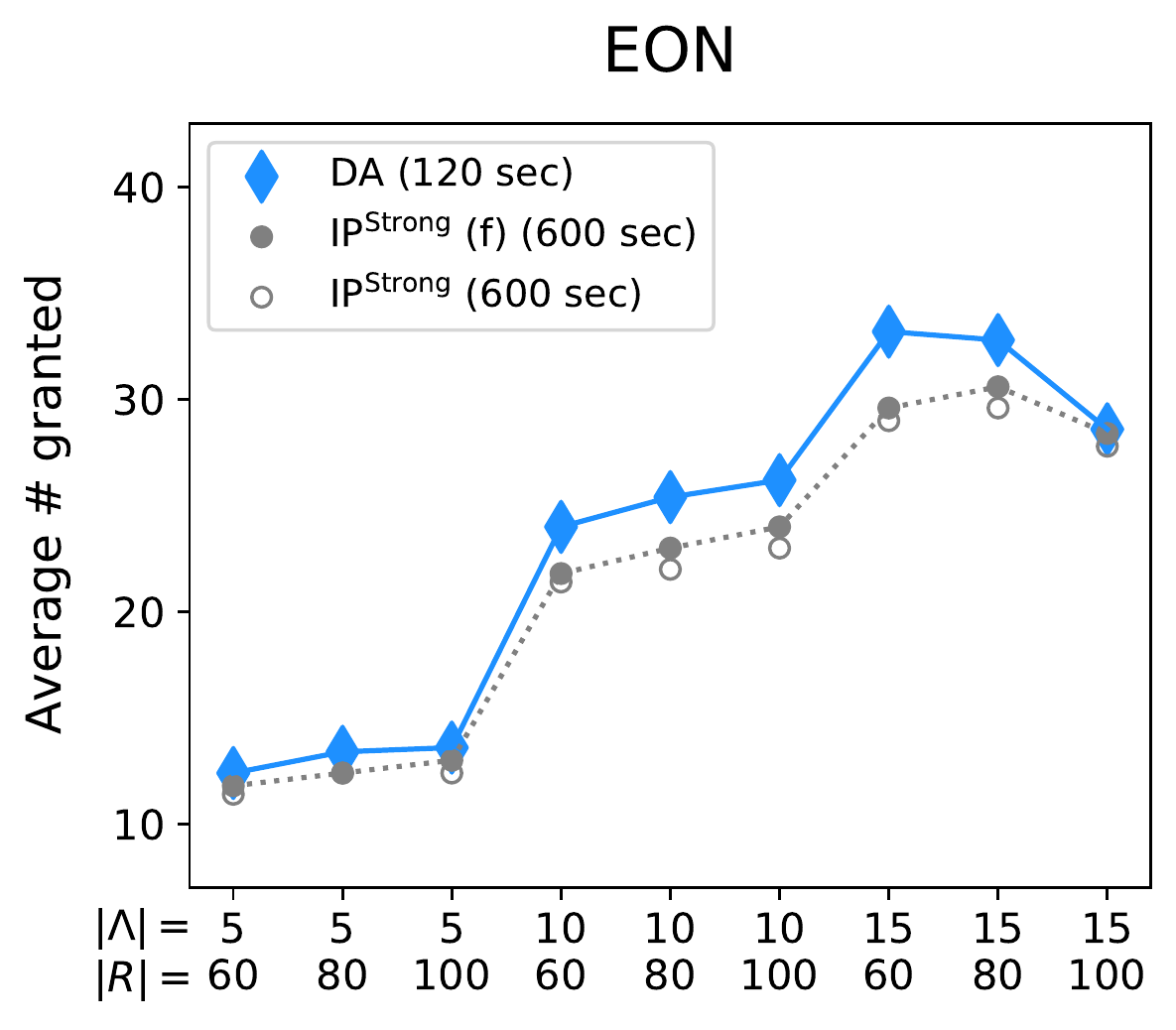}
    \end{subfigure}
     ~
    \begin{subfigure}[t]{0.35\textwidth}
	\centering
        \includegraphics[scale=0.4]{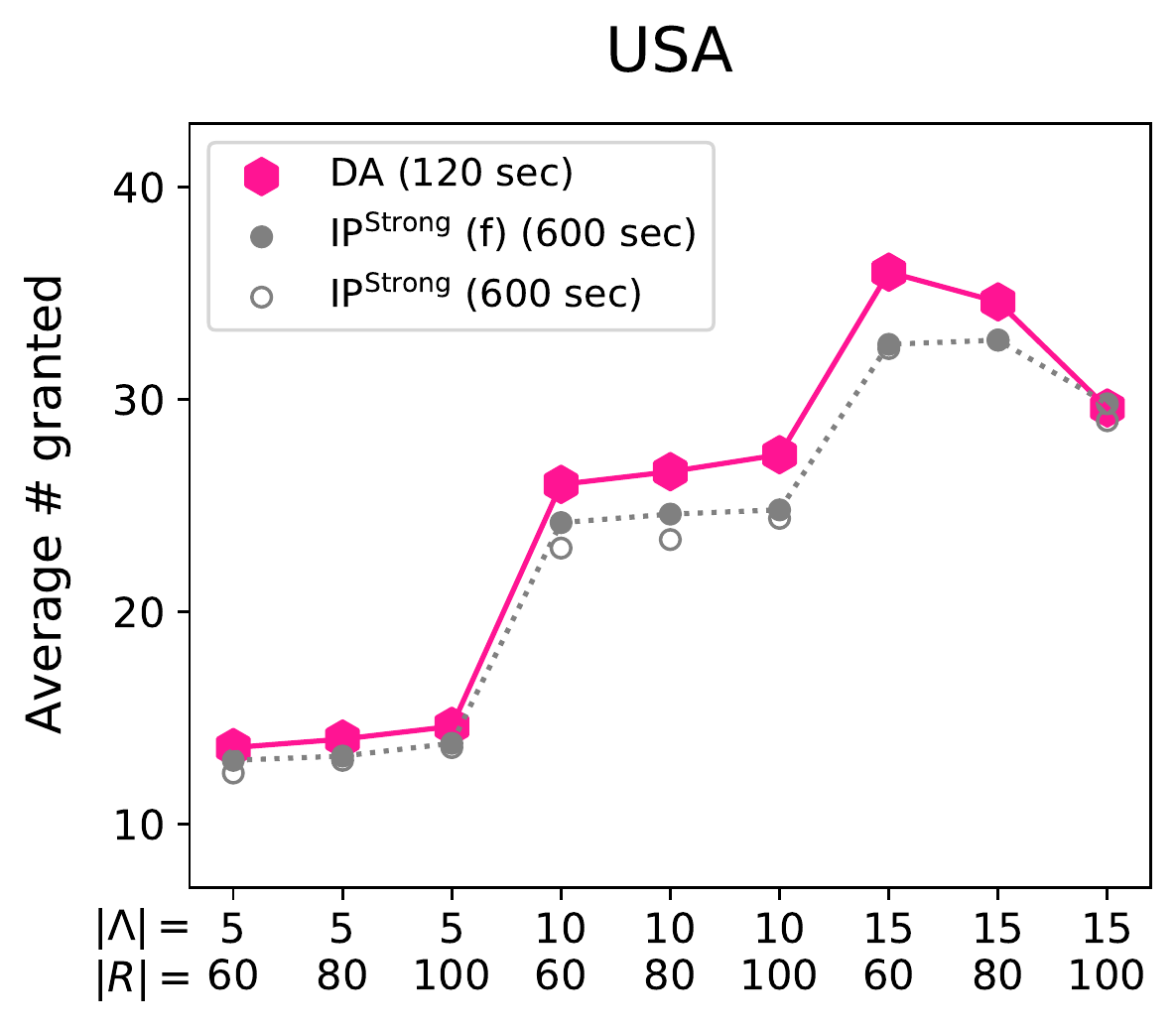}
    \end{subfigure}
     ~
    \begin{subfigure}[t]{0.35\textwidth}
	\centering
        \includegraphics[scale=0.4]{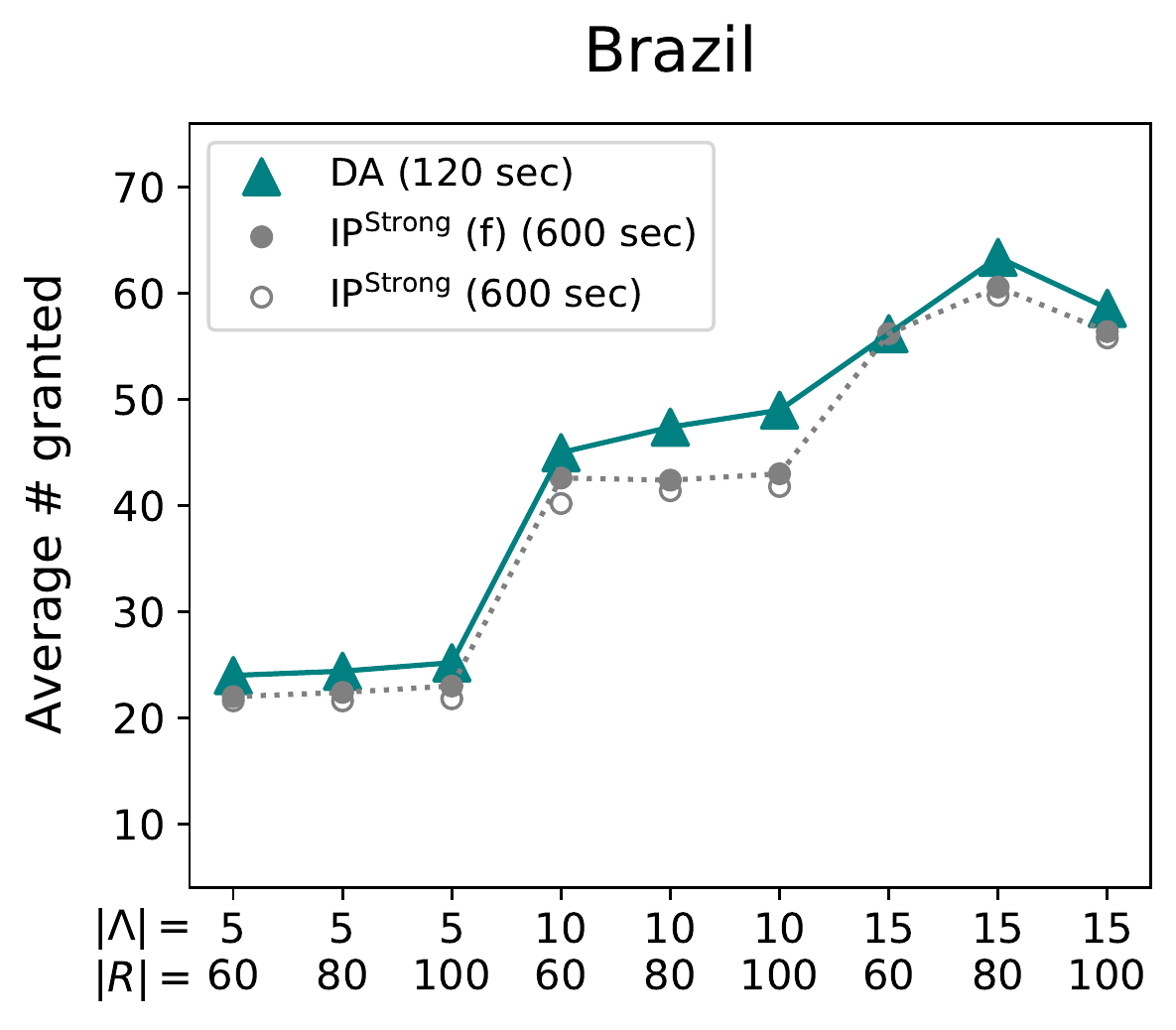}
    \end{subfigure}
    ~
    \begin{subfigure}[t]{0.35\textwidth}
	\centering
        \includegraphics[scale=0.4]{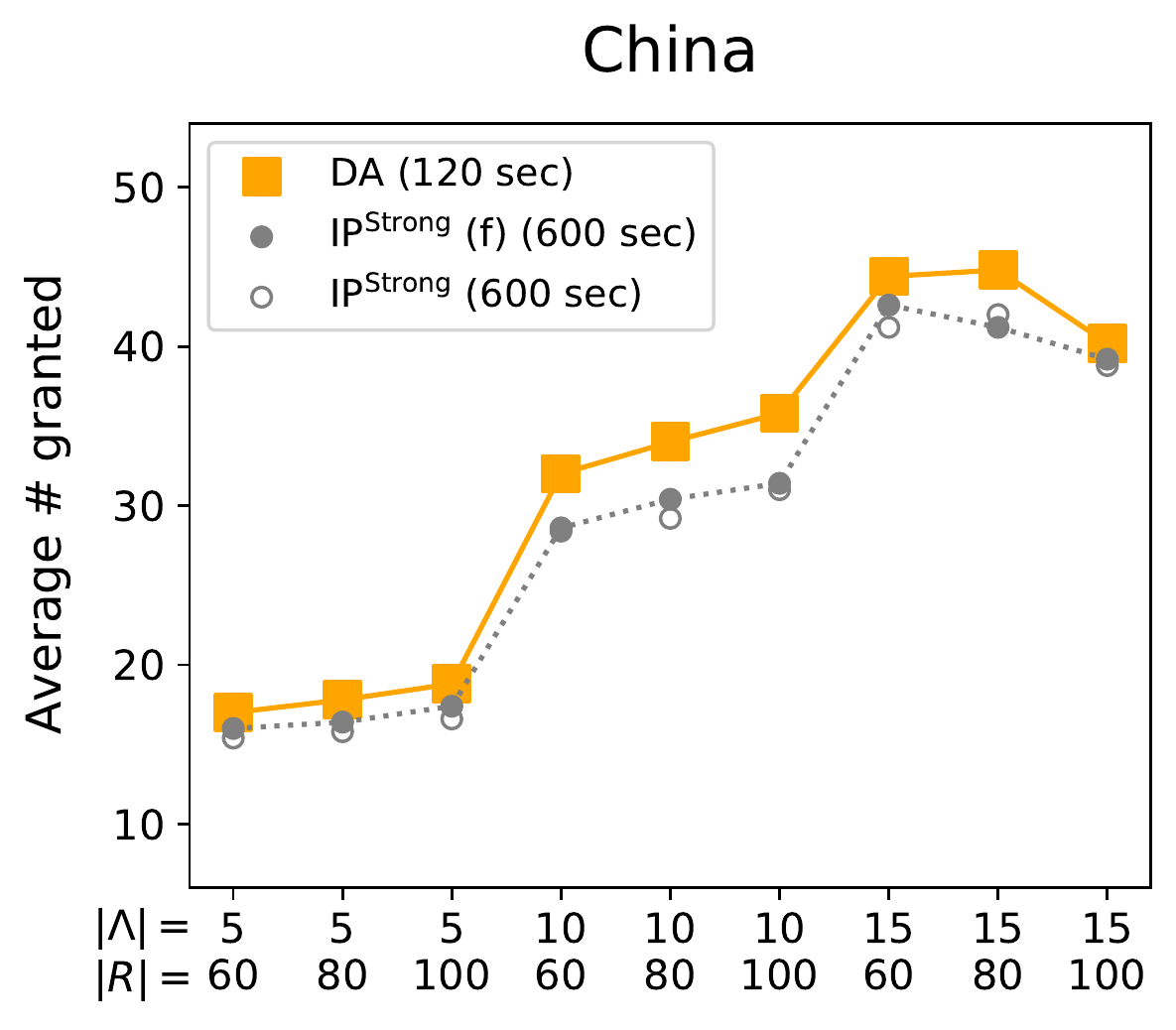}
    \end{subfigure}
    \caption{Comparison of 120-second DA and 600-second {\IPstrong} results (with and without the solver's emphasis on finding feasible solution for the latter) for every combination of number of wavelengths and requests.}
	\label{fig:DA_vs_IP}
\end{figure}

\subsection{Initial Solution and Run Time Analysis}
\label{subsec:runtime_analysis}

Having observed that GUROBI cannot much improve the solution quality despite the significant increase in the time limit, we now investigate whether feeding a DA solution to GUROBI as an initial (warm start) solution would improve its performance, and also how much time it would take for GUROBI to reach DA's performance level.
To this end, we use the set of 30 representative instances mentioned above, which covers both the extremes and the intermediate portion of the parameter space.
%we take 30 representative instances based on USA and Brazil networks, with $(|\wavelengthSet|, |\requestSet|) \in \{(5,60), (10,80), (15,100)\}$ and ten instances for each pair of parameter values, five per network. This way, the selected set of instances covers both the extremes and the intermediate portion of the parameter space.
For our initial solution analysis, we feed the DA solutions to GUROBI as a starting point in solving {\IPstrong} with a time limit of 600 seconds and emphasis on finding feasible solutions.
In order to look into the possible effect of incorporating link usage in the objective, we additionally carry out the same set of experiments by maximizing the number of granted requests only, i.e., by setting $\alpha = 0$ and $\beta = 1$.
For the run time analysis, we set the objective values of solutions from 120-second DA experiments as a stopping condition for GUROBI, along with a time limit of 7200 seconds (two hours).
Note that the objective value of a solution with a certain number of granted requests cannot be attained or surpassed without granting at least that many requests due to the way we set the value of the $\beta$ coefficient for prioritizing request granting over link usage.

Table \ref{tab:warmstart} contains the results of the above-mentioned experiments. % for each parameter pair $(|\wavelengthSet|, |\requestSet|)$ and network combination separately. 
The results of the initial solution experiments are provided under a five-column block (``{\IPstrong} (f) w/ initial sol (600 sec)"), and is divided into two groups; one for the results of the experiments with the original objective function that incorporates both the request granting and link usage components (``Weighted obj"), and the other for the case when only the request granting component of the objective is considered (``Request granting obj"). 
Each group contains the number of granted requests and used links averaged over the five instances associated with the parameter pair $(|\wavelengthSet|, |\requestSet|)$ and network in each row.
For the experiments with the original weighted objective, average optimality gap is also reported, defined as $\%~\text{gap} = \frac{\text{UB} - \text{LB}}{|\text{UB}|} \cdot 100$ where UB and LB respectively denote the upper and lower bounds for the objective value.
%...  as well as the average optimality gap percentages when the objective is in its original form (because gap values associated with only-request-granting objective are not comparable to them).
The last block of columns contains the results of the run time analysis experiments (``{\IPstrong} (f) until DA obj val (7200 sec)").
This block additionally includes the average run times in the last column (``Time"), as well as the number of granted requests GUROBI attains when the experiments are allowed to run for a full two hours, i.e., when the objective level of DA is not imposed as a stopping condition, provided in parentheses.

% Table generated by Excel2LaTeX from sheet 'Warmstart_UntilDAQuality'
\begin{table}[htbp]
  \centering
  \caption{Initial solution and run time analysis results for the selected set of representative instances.}
   \label{tab:warmstart}
   \hspace*{-0.35cm}
   \scalebox{0.67}{
    \begin{tabular}{ccccccccccccS[table-format=4.1]}
    \toprule
          & & & & \multicolumn{5}{c}{{\IPstrong} (f) w/ initial sol (600 sec)} & \multicolumn{4}{c}{{\IPstrong} (f) until DA obj val (7200 sec)} \\
          \cmidrule(lr){5-9} \cmidrule(lr){10-13} 
          & & \multicolumn{2}{c}{DA (120 sec)} & \multicolumn{3}{c}{Weighted obj} & \multicolumn{2}{c}{Request granting obj}  & \multicolumn{4}{c}{Weighted obj}  \\
    	\cmidrule(lr){3-4} \cmidrule(lr){5-7} \cmidrule(lr){8-9} \cmidrule(lr){10-13}
           \parbox{1.2cm}{\centering  $(|\wavelengthSet|, |\requestSet|)$} & \parbox{1.5cm}{\centering Network} & 
           \parbox{1.75cm}{\centering \# granted} &\parbox{1.4cm}{\centering \# links} & 
           \parbox{1.75cm}{\centering \# granted } &\parbox{1.4cm}{\centering \# links} & \parbox{1.2cm}{\centering \% gap} & 
           \parbox{1.75cm}{\centering \# granted} & \parbox{1.4cm}{\centering \# links} & 
           \parbox{2cm}{\centering \# granted} & \parbox{1.4cm}{\centering \# links} & \parbox{1.2cm}{\centering \% gap} & \parbox{1.2cm}{\centering Time} \\
    \midrule
    $(5, 60)$     & USA   &  13.6 & 227.8 & 13.6 & 227.6 & 30.6   & 13.6 & 227.8 & 13.4 (13.4) & 226.6 & 31.1 & 4343.8 \\ \vspace*{0.2cm}
    		      & Brazil &  24.0 & 379.0 & 24.0 & 379.0 & 50.4   & 24.0 & 379.0 & 22.8 (22.8) & 357.0 & 57.6 & 6742.7 \\
    $(10, 80)$   & USA   & 26.6  & 485.0 & 26.6  & 485.0 & 35.1  & 26.6 & 485.0 & 25.2 \phantom{(00.0)} & 443.2 & 42.1 & 7200.0 \\ \vspace*{0.2cm}
    		       & Brazil & 47.4  & 777.6 & 47.4  & 777.0 & 47.8  & 47.6 & 781.6 & 42.8 \phantom{(00.0)} & 706.2 & 63.7 & 7200.0 \\
    $(15, 100)$ & USA   & 29.6  & 589.8 & 30.2  & 600.8 & 12.0  & 30.2 & 606.2 & 29.6 (30.2) & 587.4 & 15.8 & 222.8 \\ 
    		       & Brazil & 58.6  & 985.8 & 58.6  & 983.4 & 17.6  & 58.6 & 985.8 & 57.2 (57.8) & 950.6 & 20.6 & 5556.8 \\ 		 
    \bottomrule
    \end{tabular}%
    }
\end{table}%

We see from Table \ref{tab:warmstart} that GUROBI is mostly unable to improve the initial solution from DA within the 600-second time limit, both in terms of granted requests and link usage.
The situation remains the same when the objective is reduced into the maximization of request granting, indicating that the link usage component in the objective does not really affect the performance, and thus a hierarchical solution approach for {\PRWA} to handle the two separate objectives would not perform any better than solving the problem with a weighted objective.
As for the run time analysis, in 66.7\% of the selected instances (20 out of 30), GUROBI cannot reach the 120-second objective value of DA after two hours. 
The time for the remaining 10 instances to reach DA's objective value is 1233 seconds on the average, an order of magnitude higher than DA's run time approximately.
Also, as a result of the full two-hour experiments, the average number of granted requests (shown in parentheses) either does not change or increases by 0.6 only.
%Moreover, the optimality gap percentages remain considerably high, around 12\% to 17\% higher than those from 600-second experiments with initial solutions.
%While the objective lower bounds from the initial solution experiments are only 0.1\% worse than those of the run time experiments, the upper bounds are 4.3\% better, signifying that the 6.2\% difference in the average optimality gap value is mainly due to the DA solutions.
In terms of average optimality gaps, the values from the initial solution experiments are 6.2\% better than those of the run time experiments.
This is mainly due to the solutions from DA, which yield 4.3\% better upper bounds than the ones from the run time experiments, while lower bounds are only 0.1\% worse.
These results indicate that DA is capable of delivering good-quality solutions in only two minutes, which typically cannot be attained by a state-of-the-art solver after two hours of run time.

\subsection{Penalty Coefficient Analysis}
\label{subsec:pentrials}

As mentioned before, penalty coefficient values rendering our QUBO formulations exact go beyond the acceptable ranges for DA.
Nevertheless, we observed in our preliminary experiments that by setting the penalty coefficient values with respect to the $\beta$ values (which we compute using Proposition \ref{prop:obj_weights}), DA can always deliver solutions that are feasible for the IP models.
Specifically, we set $\penaltyCoef = \beta + 100$, yielding IP feasible solutions throughout all of our experiments.
While testing different values for the penalty coefficient, however, we noticed that it can have a considerable impact on DA's performance, which we investigate in the sequel.
%Therefore, we next investigate the sensitivity of DA's performance to the values of penalty coefficient $\penaltyCoef$.

Figure \ref{fig:pentrials} contains the results from 120-second DA experiments with different penalty coefficient values, using the same set of 30 representative instances mentioned above.
While $x$-axes of the two plots show the penalty coefficient values as a function of the $\beta$ parameter, $y$-axes display the average number of granted requests. 
%Solid triangular and hexagonal markers stand for the values obtained from DA, and the dashed and dotted lines show the average values obtained from {\IPstrong} under 7200- and 600-second time limits, respectively.
The lower ends of the penalty coefficient values are those used in our main set of experiments, whereas the higher end is set close to the values before passing beyond the acceptable ranges for DA. % ($ \beta + 3000$).

\begin{figure}[!h]
    \centering
    \begin{subfigure}[t]{0.35\textwidth}
	\centering
        \includegraphics[scale=0.4]{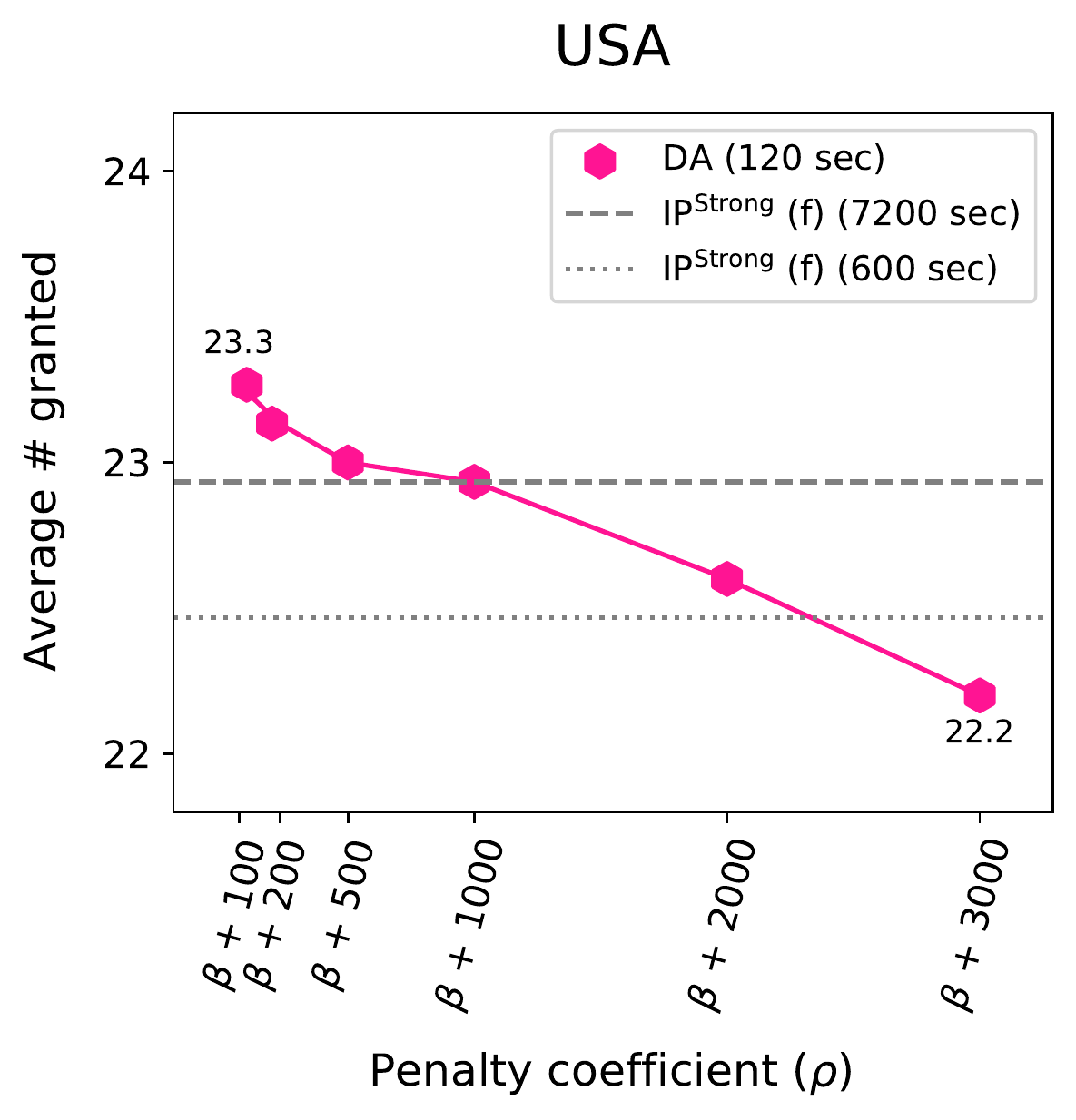}
    \end{subfigure}
	~
    \begin{subfigure}[t]{0.35\textwidth}
	\centering
        \includegraphics[scale=0.4]{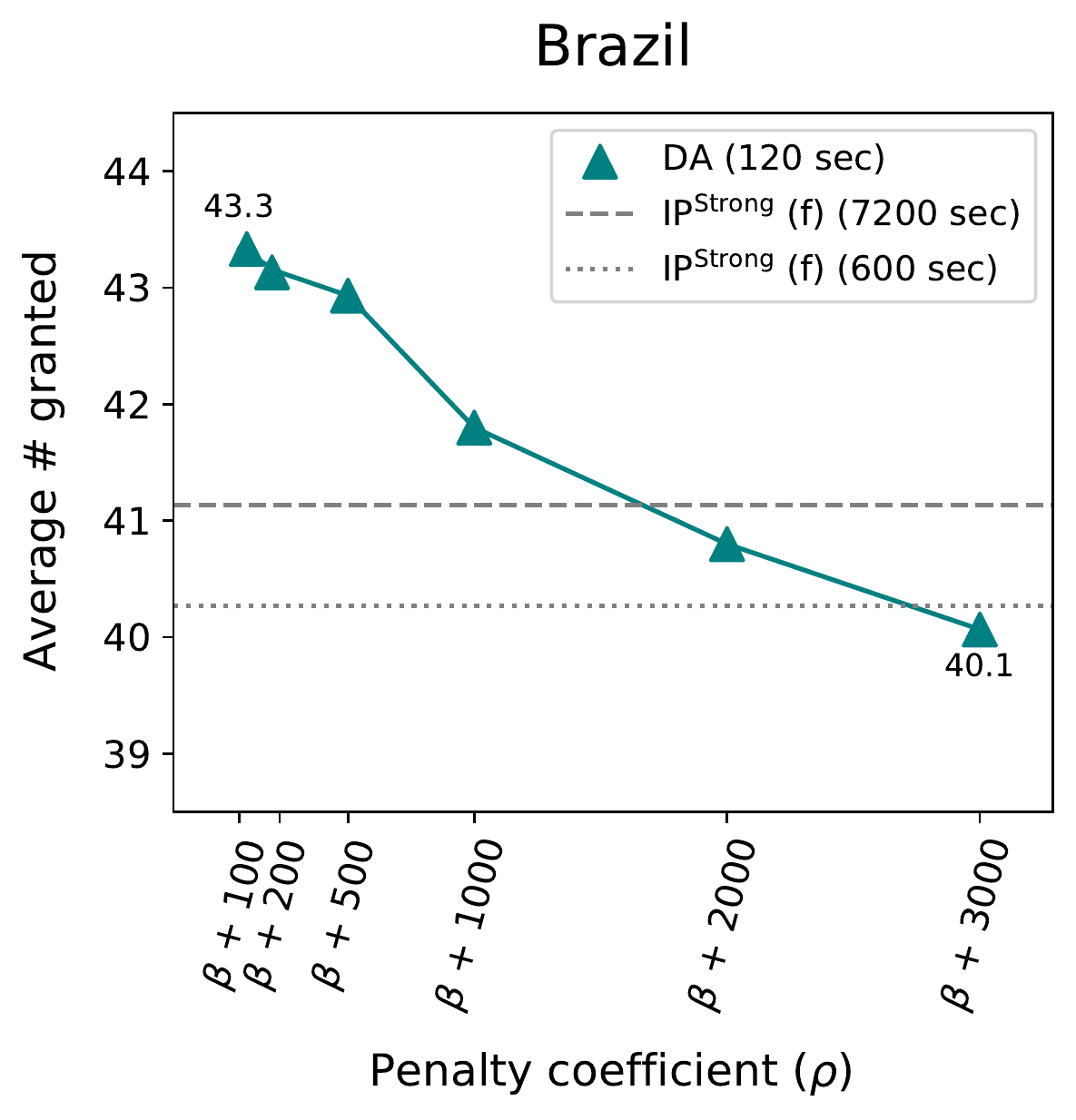}
    \end{subfigure}
	\caption{Results of the penalty coefficient trials with DA using the selected set of representative instances.}
	\label{fig:pentrials}
\end{figure}

We observe that decreasing the penalty coefficient values improve DA's performance, and the difference is more marked for the instances based on Brazil network. 
In comparison to the two-hour performance of GUROBI, not only does DA outperform with the best possible penalty coefficient values for it (i.e., when $ \penaltyCoef = \beta + 100$), it yields better or at least as good results when $ \beta + 100 \leq \penaltyCoef \leq \beta + 1000$.
Thus, even without much tuning of the penalty coefficient parameter, DA in two minutes delivers better results than a state-of-the-art solver does after hours of run times. 
%In fact, even the worst-case values from DA outperform the results from the two-hour experiments with GUROBI as IP solver. 

\subsection{Key Observations}
\label{subsec:observations}

We lastly summarize some main findings arising from our computational study. 
\BI[leftmargin=*] %, label={--}
    \I DA is the best-performing option for {\PRWA} to obtain high-quality solutions in a short amount of time {\cemph for our nontrivial set of test instances}, and established methods coupled with state-of-the-art solvers yield mostly inferior results even after two hours of run time. %, especially when the number of constraints is high.
    \I DA's performance is robust across different instances and networks.
    \I Feeding solutions obtained from DA as initial solutions to GUROBI does not improve its performance; it cannot find any better solution even after a considerable amount of run time.
    \I Solving {\PRWA} by only considering the prioritized objective of request granting and ignoring the secondary one (link usage) does not perform any better than solving it with a weighted objective.
    \I Penalty coefficient values affect DA's solution quality, but even without particularly tuned values DA can outperform the two-hour results from the state of the art. 
    %\I Network and instance characteristic is a key determinant in the performance of solvers.
    %\I To solve a QUBO model directly, DA is the outperforming option.
    %\I Finding an optimal or near-optimal solution is computationally very difficult. %, especially for cases where the number of constraints is relatively high, which typically occurs when the network resources are not sufficient to accommodate all the requests.
    %\I {\cred As a result of preliminary experiments, a hierarchical solution approach for {\PRWA} to handle the two separate objectives does not perform any better than solving the problem with a weighted objective.}
\EI

%{\cred TEXT UPDATED UP TO THIS POINT, BUT MAKE THE RESULTS SOUND MORE EXCITING, THE THRILLING EFFECT IS MISSING NOW!!}

\section{Conclusion}
\label{sec:conclusion}
In this study, we consider the routing and wavelength assignment problem with protection, {\PRWA}. Through complexity analysis and computational experiments, we show that this problem is difficult to solve both in theory and practice. 
%\sout{To address the motivating practical need of obtaining high-quality solutions in short amount of computation time,} 
We propose a viable approach, formulating {\PRWA} as a QUBO model and employing the Digital Annealer, DA, as a new promising solution technology.
We find that this approach outperforms established methods {\cemph in handling our nontrivial set of test insances} by delivering solutions within two minutes that are superior to or as good as the ones other methods yield in two hours, as well as a heuristic method from the literature. %\sout{in the majority of the test cases and is comparable otherwise.}
%We find that the emerging technology DA outperforms the considered established exact methods, in delivering mostly superior or as good solutions in only two minutes compared to two hours of run time, whereas the problem-specific heuristic fails to be competitive. 
Also considering that future generations of DA are planned to achieve megabit-class performance \citep{da1MWeb}, we believe that the proposed approach has significant potential to be utilized widely in practice. As such, future research directions involve considering large-scale cases of {\PRWA}, adaptation and application of this emerging approach to other RWA problems. 
%\sout{{\cemph and also devising other heuristic approaches and comparing with DA. }} %-- I DO NOT FEEL COMFORTABLE WITH THIS LAST ADDITION. ESEGIN AKLINA KARPUZ KABUGU DUSURMEK GIBI GELIYOR BIRAZ BANA SU ANDA...

%\setlength\bibsep{-0.95pt}
%\bibliographystyle{informs2014} % outcomment this and next line in Case 1
%\bibliography{bibliography} % if more than one, comma separated

\section*{Acknowledgement}
The authors would like to thank Fujitsu Laboratories Ltd. and Fujitsu Consulting (Canada) Inc. for providing financial support and access to Digital Annealer at the University of Toronto.

\setlength\bibsep{-1.05pt}
%\section*{References}
\bibliographystyle{elsarticle-harv}
\bibliography{bibliography}

\newpage
\appendix
%\begin{appendix}
%\EquationsNumberedThrough  
\renewcommand{\theequation}{A.\arabic{equation}}

\section{Proofs}
In this section, we provide the proofs of all the propositions, lemmas and theorems we present in the paper, as well as their statements for the sake of completeness. We number the equations in this section with a prefix of ``A.".

\subsection{Proof of Proposition \ref{prop:beta_model_onediff}}
\label{OS:proof:prop_beta_one_model}
{\it The optimization models in \eqref{eq:prioritization_optmodel0} and \eqref{eq:prioritization_optmodel1} yield the same optimal values; that is, $\LBprioExact = \LBprioOne$. This implies that the prioritization condition given in \eqref{eq:prioritization_optmodel0} can also be achieved by setting $\alpha, \beta > 0$ such that 
$\frac{\beta}{\alpha}  >  \LBprioOne $.
}
%\begin{proposition}
%\label{prop:beta_model_onediff}
%\input{IP_subfiles/prop_beta_one_model}
%\end{proposition}
%
\begin{proof}
We first show that $\LBprioOne \geq \LBprioExact$, indicating that the sufficiency of the lower bound from model \eqref{eq:prioritization_optmodel0} is implied by that of \eqref{eq:prioritization_optmodel1}, 
%We first show that model \eqref{eq:prioritization_optmodel1} yields a sufficient lower bound for $\frac{\beta}{\alpha} > \LBprioOne $, 
which can be expressed as the following logical relationship:
\begin{align}
\frac{\beta}{\alpha} 
\ > \
\LBprioOne
\ \implies \
\frac{\beta}{\alpha} 
\ > \
\LBprioExact.
\label{impl:prioritization_optmodels}
\end{align}
Suppose that the left-hand side of \eqref{impl:prioritization_optmodels} holds, which means that for any pair of feasible solutions $(\xsol, \ysol)$ and $(\xsolalt, \ysolalt)$ with $\objValGranted = \objValaltGranted + 1$, the prioritization condition given in \eqref{eq:prioritization} holds, yielding
\begin{align}
\alpha \objValLink^{\text{max}} - \beta \objValGranted
\ < \
\alpha \objValaltLink^{\text{min}} - \beta \objValaltGranted.
\label{eq:prioritized_pair1}
\end{align}
Let $(\xsolaltalt, \ysolaltalt)$ be a feasible solution with $\objValaltaltGranted = \objValGranted + 1$. 
Then, the following holds by assumption:
\begin{align}
\alpha \objValaltaltLink^{\text{max}} - \beta \objValaltaltGranted 
\ &< \ 
\alpha \objValLink^{\text{min}} - \beta \objValGranted. \label{eq:prioritized_pair2}
\intertext{Combining \eqref{eq:prioritized_pair1} and \eqref{eq:prioritized_pair2}, we get}
\alpha \objValaltaltLink^{\text{min}} - \beta \objValaltaltGranted 
\ < \ 
\alpha \objValaltaltLink^{\text{max}} - \beta \objValaltaltGranted 
\ &< \ 
\alpha \objValLink^{\text{max}} - \beta \objValGranted
\ < \
\alpha \objValaltLink^{\text{min}} - \beta \objValaltGranted. \label{eq:inductive_chain}
\end{align}
This shows that $\objValaltaltGranted = \objValaltGranted + 2 \implies \objValaltalt < \objValalt$. 
Since the chain of inequalities in \eqref{eq:inductive_chain} can be extended to any feasible solution $(\xsolaltalt, \ysolaltalt)$ with $\objValaltaltGranted = \objValaltGranted + \grantedDiff$ for $\grantedDiff \geq 2$, we conclude that the relationship in \eqref{impl:prioritization_optmodels} holds and hence
$\LBprioOne \geq \LBprioExact $.
Moreover, the feasible region of the model in \eqref{eq:prioritization_optmodel0} contains that of \eqref{eq:prioritization_optmodel1}, thus $\LBprioOne \leq \LBprioExact$, which establishes $\LBprioExact = \LBprioOne$.
\hfill %\Halmos 
\end{proof}

\subsection{Proof of Proposition \ref{prop:obj_weights}}
\label{OS:proof:prop_obj_weights}
{\it %\begin{proposition}[IP objective weight selection]
%\label{prop:obj_weights}
Selecting $\alpha, \beta > 0$ such that 
\begin{align}
    \frac{\beta}{\alpha} \ &> \ |\requestSet| \ (\maxLPpairLength - 2) \ + \ 2  \tag{\ref{eq:obj_coef_ratio_suff_lb}}
\end{align}
prioritizes request granting over link usage in \eqref{eq:OF} for any feasible solution to the IP, i.e., solutions accepting more requests yield lower objective values, where ${\maxLPpairLength = \max\limits_{\requestInd \in \requestSet}\left\{ \max\limits_{\workingInd \in \workingSet^\requestInd} \big\{ \length^{\requestInd}_{\workingInd} \big\} + \max\limits_{\protectionInd \in \protectionSet^\requestInd}\big\{ \length^{\requestInd}_{\protectionInd} \big\} \right\}}$.
%where $\maxLPpairLength = \max_{\requestInd \in \requestSet}\left\{ \length^{\requestInd}_{wmax} +  \length^{\requestInd}_{pmax} \right\}$ with $ \length^{\requestInd}_{wmax} = \max_{\workingInd \in \workingSet^\requestInd}\left\{ \length^{\requestInd}_{\workingInd} \right\}$ and  $ \length^{\requestInd}_{pmax} = \max_{\protectionInd \in \protectionSet^\requestInd}\left\{ \length^{\requestInd}_{\protectionInd} \right\}$. 
%\end{proposition}
%
}
% \begin{proposition}[IP objective weight selection]
% \label{prop:obj_weights}
% Selecting $\alpha, \beta > 0$ such that 
% \begin{align}
%     \frac{\beta}{\alpha} \ &> \ |\requestSet| \ (\maxLPpairLength - 2) \ + \ 2 \label{eq:obj_coef_ratio_suff_lb}
% \end{align}
% prioritizes request granting over link usage in \eqref{eq:OF} for any feasible solution to the IP, i.e., solutions accepting more requests yield lower objective values, where ${\maxLPpairLength = \max\limits_{\requestInd \in \requestSet}\left\{ \max\limits_{\workingInd \in \workingSet^\requestInd} \big\{ \length^{\requestInd}_{\workingInd} \big\} + \max\limits_{\protectionInd \in \protectionSet^\requestInd}\big\{ \length^{\requestInd}_{\protectionInd} \big\} \right\}}$.
% \end{proposition}
%
%
\begin{proof}
Suppose that the parameters $\alpha, \beta > 0$ satisfy condition \eqref{eq:obj_coef_ratio_suff_lb}. 
\begin{comment}
We want to show the following logical relationship: 
\begin{align*}
    \objValGranted \ = \ \objValaltGranted + 1 \ \implies \ \objVal \ < \ \objValalt %\label{impl:wts_prior_rhs}
\end{align*}
\end{comment}
Let $(\xsol, \ysol)$ and $(\xsolalt, \ysolalt)$ be two feasible solutions 
%feasible to the IP 
that satisfy $\objValGranted > \objValaltGranted$.
By Proposition \ref{prop:beta_model_onediff}, it suffices to show that the prioritization condition in \eqref{eq:prioritization} holds when $\objValGranted = \objValaltGranted + 1$, i.e., we want to show the following logical relationship:
\begin{align}
    \objValGranted \ = \ \objValaltGranted + 1 \ \implies \ \alpha \objValLink^{\text{max}} - \beta \objValGranted  \ < \ \alpha \objValaltLink^{\text{min}} - \beta \objValaltGranted  \label{eq:wts_prior}
\end{align}

Assume for a contradiction that we have $\objValGranted = \objValaltGranted + 1$ and $ \alpha \objValLink^{\text{max}} - \beta \objValGranted  \ \geq \ \alpha \objValaltLink^{\text{min}} - \beta \objValaltGranted$, which can be equivalently written as
\begin{align}
\beta \objValGranted \ - \ \beta \objValaltGranted \ \leq \ \alpha \objValLink^{\text{max}} \ - \ \alpha \objValaltLink^{\text{min}}. \nonumber
\intertext{Rearranging the terms, we obtain}
\frac{\beta}{\alpha} 
\ \leq \ 
\frac{\objValLink^{\text{max}} \ - \ \objValaltLink^{\text{min}}}{\objValGranted \ - \ \objValaltGranted} 
\ = \
\objValLink^{\text{max}} \ - \ \objValaltLink^{\text{min}}, \label{eq:obj_coef_ratio_ub}
\end{align}
because $\alpha > 0$ and $\objValGranted - \objValaltGranted = 1$. Since each lightpath is comprised of at least one link, we have 
\begin{align}
    \objValaltLink^{\text{min}} \ \geq \ 2 \objValaltGranted. \label{eq:fmin_lb}
\end{align}
Moreover, in the worst case, the longest working and protection lightpaths will be used for each granted request, yielding
\begin{align}
    \objValLink^{\text{max}} 
    \ \leq \
    \objValGranted \cdot \max\limits_{\requestInd \in \requestSet}\left\{ \max\limits_{\workingInd \in \workingSet^\requestInd} \big\{ \length^{\requestInd}_{\workingInd} \big\} \ + \ \max\limits_{\protectionInd \in \protectionSet^\requestInd}\big\{ \length^{\requestInd}_{\protectionInd} \big\} \right\}
    \ \eqqcolon \ 
    \objValaltGranted \ \maxLPpairLength. \label{eq:fmax_ub}
\end{align}
Combining  \eqref{eq:obj_coef_ratio_ub}, \eqref{eq:fmin_lb} and \eqref{eq:fmax_ub}, we obtain
\begin{align*}
\frac{\beta}{\alpha} 
\ &\leq \ 
\objValLink^{\text{max}} \ - \ \objValaltLink^{\text{min}} 
\ \leq \   
\objValGranted \ \maxLPpairLength \ - \  2 \objValaltGranted. 
\intertext{Plugging in $\objValGranted - 1  = \objValaltGranted$, we get}
\frac{\beta}{\alpha} 
\ &\leq \ 
\objValGranted \ \maxLPpairLength \ - \  2 (\objValGranted - 1)
\ = \
\objValGranted ( \maxLPpairLength -  2) \ + \ 2 
\end{align*}
Since $\objValGranted \leq |\requestSet|$, we get
\begin{align*}
\frac{\beta}{\alpha} 
\ &\leq \ 
|\requestSet| \ ( \maxLPpairLength -  2) \ + \ 2 ,
\end{align*}
which contradicts with our premise in \eqref{eq:obj_coef_ratio_suff_lb}, and thus proves that for $\objValGranted = \objValaltGranted + 1$, satisfying the suggested condition in \eqref{eq:obj_coef_ratio_suff_lb}  guarantees that the prioritization condition in \eqref{eq:prioritization} holds, and thus concludes the proof. \hfill %\Halmos
\end{proof}

\subsection{Proof of Proposition \ref{prop:obj_weights_tightLB}}
\label{OS:proof:prop_obj_weights_tightLB}
{\it % \begin{proposition}[Tight example for the weight selection]
% \label{prop:obj_weights_tightLB}
There exist {\PRWA} instances for which the lower bound provided in Proposition \ref{prop:obj_weights} is %tight 
necessary to prioritize request granting over link usage.
%for binary solutions.
% \end{proposition}
}
% \begin{proposition}[Tight example for the weight selection]
% \label{prop:obj_weights_tightLB}
% There exist {\PRWA} instances for which the lower bound provided in Proposition \ref{prop:obj_weights} is %tight 
% necessary to prioritize request granting over link usage for binary solutions.
% \end{proposition}
% %
%
\begin{proof}
%the following input must remain local
%\input{IP_subfiles/macros_obj_weights}
\newcommand{\len}{\ell}

Consider an {\PRWA} network $G$ being comprised of a pair of source and destination nodes $s$ and $t$ that are linked through two link-disjoint paths of length $\len_a$ and $\len_b$. 
Suppose that there is only one request, i.e., $|\requestSet| = 1$, which is between nodes $s$ and $t$, and the two distinct paths of length $\len_a$ and $\len_b$ together with some wavelength $\wavelength$ constitute the working and protection lightpaths for it, respectively.

Assume that $\alpha, \beta > 0$ are selected in such a way that %request granting is  prioritized over link usage. That is, 
for any pair of feasible solutions $(\xsol, \ysol)$ and $(\xsolalt, \ysolalt)$ with $\objValGranted > \objValaltGranted$, the prioritization condition in \eqref{eq:prioritization} holds.
Under this assumption, we want to show that the selected $\alpha, \beta > 0$ must satisfy the lower bound suggested in Proposition \ref{prop:obj_weights}, which for the above-mentioned class of instances becomes
\begin{align}
\frac{\beta}{\alpha} 
\ > \ 
|\requestSet| \ (\maxLPpairLength - 2) \ + \ 2
\ = \ \len_a \ + \ \len_b. \label{eq:obj_coef_ratio_lb_tight_ex}
\end{align}
%
%We assume for a contradiction that $\tfrac{\beta}{\alpha} \leq \len_a + \len_b$, and consider two solutions that satisfy

Let $(\xsol, \ysol) = (1,1)$ and $(\xsolalt, \ysolalt) = (0,0)$, giving $\objValGranted = 1$ and $\objValaltGranted = 0$. Then, by \eqref{eq:prioritization_v0}, we have
$$\alpha \ (\len_a + \len_b) \ - \ \beta \cdot 1 \ < \ \alpha \cdot 0 \ - \ \beta \cdot 0,$$ 
giving the desired necessary condition of $\frac{\beta}{\alpha}  \ > \ \len_a + \len_b$. 
\hfill %\Halmos
\end{proof}

\subsection{Proof of Lemma \ref{lem:np}}
\label{OS:proof:lem_np}
{\it % \begin{lemma}[Verifiability]
% \label{lem:np}
{\DPRWAreq} is in NP.
% \end{lemma}
}
% \begin{lemma}[Verifiability]
% \label{lem:np}
% {\DPRWAreq} is in NP.
% \end{lemma}
%
\begin{proof}
\newcommand{\PRWAreqSol}{\mathcal{S}}
\newcommand{\numGranted}{k}

Given an {\PRWAreq} instance  $\PRWAreqInst =  \left(G, \requestSet, \{\workingSet^\requestInd\}_{\requestInd \in \requestSet}, \{ \protectionSet^\requestInd \}_{r \in \requestSet}  \right) $, where $G = (V, \linkSet) $ is a directed graph representing the optical network, $\requestSet$ is the set of requests defined between pairs of source and destination nodes, and $\workingSet^\requestInd$ and $\protectionSet^\requestInd$ are respectively the set of  working and protection lightpaths for request $\requestInd \in \requestSet$, and a solution $\PRWAreqSol$ to instance $\PRWAreqInst$, we will show that we can verify in time polynomial in the size of $\PRWAreqInst$ whether or not $\PRWAreqSol$ properly grants $\numGranted$ requests for some given integer $\numGranted \geq 0$.
%some introduction for the following list goes here
In what follows, we provide the complexity of each step in checking if solution $\PRWAreqSol$ satisfies the constraints of the {\PRWAreq} problem and whether the number of granted requests is at least $\numGranted$. 
\begin{enumerate}
    \item For each request $\requestInd \in \requestSet$, we check if it is assigned (i) both a working and a protection lightpath or (ii) no lightpaths at all, which can be done by going over all working and protection lightpaths each time, which amounts to $ \bigO \left( \sum_{\requestInd \in \requestSet} \left(|\workingSet^\requestInd| + |\protectionSet^\requestInd| \right) \right)$ time in total. 
    If the number of granted requests is strictly less than $\numGranted$, the verification procedure terminates by concluding that solution $\PRWAreqSol$ does not grant $\numGranted$ requests. Otherwise, we continue verifying solution $\PRWAreqSol$ with the following steps.
    \item  For each granted request, we check if the assigned working and protection lightpaths are link-disjoint by going over all tuples in conflict set $\conflictSet_1$.
    In the worst case, all possible pairs of working and protection lightpaths will be contained in $\conflictSet_1$ for each request $\requestInd \in \requestSet$, which makes the time complexity of this step $\bigO \left( \sum_{\requestInd \in \requestSet} |\workingSet^{\requestInd}| \cdot |\protectionSet^{\requestInd}| \right) $.
    \item For the lightpaths used in solution $\PRWAreqSol$, we check if each pair having the same wavelength is link-disjoint. To this end, we go through the tuples in conflict sets $\conflictSet_2$, $\conflictSet_3$ and $\conflictSet_4$. In the worst case, all possible tuples will be contained in these sets, which yields the following complexities: 
    \BI
        \I If the pair is comprised of a working and protection lightpath, we search through the tuples in $\conflictSet_2$, amounting to $\bigO \left( \sum_{\requestInd_1 \in \requestSet} \sum_{\requestInd_2 \in \requestSet} (|\workingSet^{\requestInd_1}| \cdot |\protectionSet^{\requestInd_2}| \right) $ time in total.
        \I If both lightpaths in the pair are working ones, we go through the tuples in $\conflictSet_3$, which takes $\bigO \left( \sum_{\requestInd_1 \in \requestSet} \sum_{\requestInd_2 \in \requestSet} (|\workingSet^{\requestInd_1}| \cdot |\workingSet^{\requestInd_2}| \right) $ time in total.
        \I If both in the pair are protection lightpaths, we search through the tuples in $\conflictSet_4$, taking a total of $\bigO \left( \sum_{\requestInd_1 \in \requestSet} \sum_{\requestInd_2 \in \requestSet} (|\protectionSet^{\requestInd_1}| \cdot |\protectionSet^{\requestInd_2}| \right) $ time.
    \EI
\end{enumerate}
The time it takes to perform each one of the three steps listed above is polynomial in the size of the instance $\PRWAreqInst$.
Summing up the complexities of individual steps, we conclude that the overall complexity of verifying a given solution is polynomial in the size of $\PRWAreqInst$ by the composition property of polynomials. 
Hence, {\DPRWAreq} is in NP.
\hfill % \Halmos
\end{proof}

\begin{comment}
{\cpink
SHOWING THAT {\DPRWAreq} in NP:
Given a solution with $k$ granted requests, can we check its correctness in poly time?
We have $k_w$ working and $k_p$ protection lightpaths given, 
and we need to make sure that 

%0. Check if  $k_w = k_p = k$: constant time
1. For each $r \in R$, check if both a working and a protection path is assigned or none assigned: $O(\sum_{r \in R}(|W^r|+|P^r|))$ time

2. For each request, the associated $w$ and $p$ are link-disjoint: check all tuples in $\conflictSet_1$ in $O(\sum_{r \in R} |W^{r}| \cdot |P^{r}| ) $  

3. Each pair of selected lightpaths with the same wavelength should be link-disjoint: check all tuples in $\conflictSet_2, \conflictSet_3, \conflictSet_4$

$\conflictSet_2$: $O(\sum_{r_1 \in R} \sum_{r_2 \in R} (|W^{r_1}| \cdot |P^{r_2}|) $

$\conflictSet_3$: $O(\sum_{r_1 \in R} \sum_{r_2 \in R} (|W^{r_1}| \cdot |W^{r_2}|) $

$\conflictSet_4$: $O(\sum_{r_1 \in R} \sum_{r_2 \in R} (|P^{r_1}| \cdot |P^{r_2}|) $
}
\end{comment}

\subsection{Proof of Theorem \ref{thm:npcomplete}}
\label{OS:proof:thm_npcomplete}
{\it % \begin{theorem}[Complexity]
% \label{thm:npcomplete}
	{\PRWAreq} is NP-hard.
% \end{theorem}
}
% \newcommand{\mssnode}{n}
% \newcommand{\wnode}{u}
% \newcommand{\pnode}{v}
% \newcommand{\uNodeOfEdge}{\mathcal{u}}
% \newcommand{\vNodeOfEdge}{\mathcal{v}}
% \newcommand{\exMSSinst}{G_s}
% \newcommand{\exPRWAnetwork}{G}
% \newcommand{\exConfSet}{\conflictSet}

% \begin{theorem}[Complexity]
% \label{thm:npcomplete}
% 	{\PRWAreq} is NP-hard.
% \end{theorem}
%
%
\begin{proof}
    We prove the hardness of {\PRWAreq} through a reduction from {\mss}.
    First, we describe a polynomial time reduction from a given {\mss} instance $G_s = \left( V_s, E_s \right)$ into an {\PRWAreq} instance $\PRWAreqInst =  \left(G, \requestSet, \{\workingSet^\requestInd\}_{r \in \requestSet}, \{ \protectionSet^\requestInd \}_{r \in \requestSet}  \right) $. 
    %where $G = (V, E) $ is a graph representing the optical network, $\requestSet$ is the set of requests defined between pairs of source and destination nodes, while $\workingSet^\requestInd$ and $\protectionSet^\requestInd$ are respectively the set of  working and protection lightpaths for request $r \in \requestSet$.
    Afterwards, we show that solving {\dmss} in $G_s$ is equivalent to solving {\DPRWAreq} in $\PRWAreqInst$; that is, finding a stable set of size at least $k$ in $G_s$ is equivalent to granting at least $k$ requests in $\PRWAreqInst$, for some given integer $1 \leq k \leq |V_s|$.

    {\it Construction.} Given an {\mss} instance $G_s = \left( V_s, E_s \right)$, we begin with defining the request set $\requestSet$ and an initial construction of $G$ along with initial working and protection lightpath sets $\workingSet^\requestInd$ and $\protectionSet^\requestInd$ for each $\requestInd \in \requestSet$.
    Then, we form the conflict sets $\conflictSet_1, \conflictSet_2, \conflictSet_3, \conflictSet_4$ using the edge set $E_s$ of the {\mss} instance, and modify $G$ and lightpath sets into their final form in accordance with the conflict sets. More specifically, we first create all link-disjoint lightpaths, and then modify them along with the network step by step, where in each step we create one distinct conflict from the conflict sets by modifying the involved lightpaths to have a common newly created link. 
Our reduction yields a simple directed graph $G$ (i.e., we do not introduce multiple links between the same ordered pair of nodes although it is allowed), where there exist a unique working and protection lightpath for each request, the lightpaths pairwise have at most one link in common, and all of them have the same wavelength.
    %and lightpaths having the same wavelength, pairwise sharing at most one distinct link, and each uniquely belonging to only one working or protection set of a request. 
    
    We divide our reduction procedure into four steps as explained in detail below.

    \begin{enumerate}
		\item {\it Requests.} We create $|V_s|$ requests, i.e., we set $\requestSet = \{ 1,\hdots, |V_s| \} $. Specifically, denoting the nodes in $V_s$ by $\mssnode_\requestInd , \requestInd \in \requestSet $, for each {\mss} node $\mssnode_\requestInd $, we create a pair of source and destination nodes $(s^\requestInd, t^\requestInd)$ for the optical network $G$ and request $\requestInd$ associated with it.
	
		%Namely,
		%\begin{align*}
		%\requestSet &= \left \{(s^\requestInd, t^\requestInd) \colon  r = 1,\hdots, |V_s| \right \}.
        %\requestSet &= \left \{\requestInd \colon \requestInd = (s^\requestInd, t^\requestInd), \  \mssnode_\requestInd \in V_s \right \}.
        %\end{align*}
        
        \smallskip
        
        \item {\it Initial network and lightpath sets.} For each request $\requestInd \in \requestSet$, we create two nodes $\wnode^\requestInd_1$ and $\pnode^\requestInd_1$, and generate two lightpaths between $s^\requestInd$ and $t^\requestInd$, namely a working lightpath $w \in \workingSet^\requestInd$ with 
        $$\linkSet[\workingInd] = \left\{ (s^\requestInd, \wnode^\requestInd_1), (\wnode^\requestInd_1, t^\requestInd) \right\} \ \text{and} \ \wavelengthOfLightpath(\workingInd) = \wavelength$$
        and a protection lightpath $p \in \protectionSet^\requestInd$ with
        $$\linkSet[\protectionInd] = \left\{ (s^\requestInd, \pnode^r_1), (\pnode^r_1, t^\requestInd) \right\} \ \text{and} \  \wavelengthOfLightpath(\protectionInd) = \wavelength,$$
        
        \noindent where $\linkSet[\ell]$ and $\wavelengthSet(\ell)$ respectively represent the set of links a given lightpath $\ell$ contains and the wavelength associated with it, as mentioned in Section \ref{sec:IP-formulation}.

        In words, the working and protection lightpaths for each request $\requestInd \in \requestSet$ and further all the lightpaths are initially link-disjoint, contain two links, and possess the same (arbitrary) wavelength $\wavelength$ (e.g., $\wavelength = 1$).
        
		Figure \ref{fig:nphard_mss_to_init_paths} illustrates an example {\mss} instance and the way a corresponding {\PRWAreq} network looks after performing steps 1 and 2 explained above, i.e., after constructing the request set $\requestSet$ and the initial working and protection lightpath sets $\workingSet^\requestInd$ and $\protectionSet^\requestInd$ for each $\requestInd \in \requestSet$. 
		The nodes of the {\mss} instance $\exMSSinst$ in part (i) of Figure \ref{fig:nphard_mss_to_init_paths} are highlighted with different colors, and the same set of colors are used to indicate the source and destination nodes of the corresponding requests in the {\PRWAreq} network $\exPRWAnetwork$ in part (ii). The paths through nodes with $\wnode$ labels represent those of the working lightpaths, whereas the ones through $\pnode$ labels show those of the protection lightpaths.

        \begin{figure}[!h]
\centering
\tikzset{main_node/.style={circle,draw,line width=0.75pt,minimum size=0.75cm,inner sep=0pt},}
\tikzset{node1/.style={circle,draw,line width=0.75pt,minimum size=0.75cm,inner sep=0pt,fill = red},}
\tikzset{node2/.style={circle,draw,line width=0.75pt,minimum size=0.75cm,inner sep=0pt,fill = yellow},}
\tikzset{node3/.style={circle,draw,line width=0.75pt,minimum size=0.75cm,inner sep=0pt,fill = blue!60!white},}
\tikzstyle{edge} = [draw,line width=1pt,-]
\tikzstyle{arclr} = [draw,line width=1pt,-latex]
\tikzstyle{arcrl} = [draw,line width=1pt,latex-]
\tikzstyle{highlighted edge} = [draw,line width=6pt,-,green!70!black, opacity=0.3]
\scalebox{0.65}{
\begin{tikzpicture}

    \node[node1,minimum size=0.95cm] (u1) {\large $\mssnode_1$};
    \node[node2,minimum size=0.95cm] (u2) [below left = 0.5 cm and 0.5 cm of u1] {\large $\mssnode_2$};
    \node[node3,minimum size=0.95cm] (u3) [below right = 0.5 cm and 0.5 cm of u1]  {\large $\mssnode_3$};
    
    \path[edge]
         (u1) to (u2)
         (u1) to (u3);
    
    \begin{scope}[xshift=9cm, yshift=1cm]
        \node[node1] (s_1) {$s^1$};
        \node[main_node] (w_11) [above right = 0.25 cm and 1.25 cm of s_1] {$\wnode^1_1$};
        \node[main_node] (p_11) [below right = 0.25 cm and 1.25 cm of s_1]  {$\pnode^1_1$};
        \node[node1] (t_1) [below right = 0.25 cm and 1.25 cm of w_11]  {$t^1$};     
    \end{scope}
    
     \draw[arclr] (s_1) edge[bend left=20]  (w_11);
     \draw[arclr] (w_11) edge[bend left=20] (t_1);
     \draw[arclr] (s_1) edge[bend right=20]  (p_11);
     \draw[arclr] (p_11) edge[bend right=20] (t_1);
    
    \begin{scope}[xshift=6cm, yshift=-2cm]
        \node[node2] (s_2) {$s^2$};
        \node[main_node] (w_21) [above right = 0.25 cm and 1.25 cm of s_2] {$\wnode^2_1$};
        \node[main_node] (p_21) [below right = 0.25 cm and 1.25 cm of s_2]  {$\pnode^2_1$};
        \node[node2] (t_2) [below right = 0.25 cm and 1.25 cm of w_21]  {$t^2$};
             
        \draw[arclr] (s_2) edge[bend left=20]  (w_21);
     	\draw[arclr] (w_21) edge[bend left=20] (t_2);
     	\draw[arclr] (s_2) edge[bend right=20]  (p_21);
     	\draw[arclr] (p_21) edge[bend right=20] (t_2);
             
    \end{scope}
    
    \begin{scope}[xshift=12cm, yshift=-2cm]
        \node[node3] (s_3) {$s^3$};
        \node[main_node] (w_31) [above right = 0.25 cm and 1.25 cm of s_3] {$\wnode^3_1$};
        \node[main_node] (p_31) [below right = 0.25 cm and 1.25 cm of s_3]  {$\pnode^3_1$};
        \node[node3] (t_3) [below right = 0.25 cm and 1.25 cm of w_31]  {$t^3$};
        
          \draw[arclr] (s_3) edge[bend left=20]  (w_31);
          \draw[arclr] (w_31) edge[bend left=20] (t_3);
          \draw[arclr] (s_3) edge[bend right=20]  (p_31);
          \draw[arclr] (p_31) edge[bend right=20] (t_3);
    \end{scope}
       
    \node (info_node1) [below = 3.5 cm of u1] {(i) };
    \node (info_node2) [below = 3.75 cm of p_11] {(ii)};

\end{tikzpicture}
}
\caption{(i) An {\mss} instance $\exMSSinst$, (ii) the initial incomplete form of the corresponding {\PRWAreq} network $\exPRWAnetwork$. }
\label{fig:nphard_mss_to_init_paths}
\end{figure}	
		
		\item {\it Conflict sets.} Next, we define the conflict sets $\conflictSet_1, \ldots, \conflictSet_4$, which we will utilize to modify the initial form of the network $G$. 
		Specifically, for each pair of requests corresponding to an edge of $G_s$, we introduce different conflict tuples for $\conflictSet_i$'s, as detailed below:
        \medskip
        
		    (i) \ We define the working and protection lightpaths of any fixed request $\requestInd \in \requestSet$ to be link-disjoint. So, we have  
		    \begin{align*}
	        \conflictSet_1 \ &= \ \varnothing. 
		    \intertext{(ii) \ For each edge of $G_s$, we want the working and protection lightpaths of the associated requests to have a link in common. 
		    In particular, for each edge $ \{\mssnode_{\requestInd_1}, \mssnode_{\requestInd_2}\} \in E_s $, we want the lightpath pairs $(a)$ $ \workingInd \in \workingSet^{\requestInd_1}$, $\protectionInd \in \protectionSet^{\requestInd_2} $, and $(b)$ $ \workingInd \in \workingSet^{\requestInd_2}$, $\protectionInd \in \protectionSet^{\requestInd_1} $ to share a link.
		    As such, we define }
		    \conflictSet_2 \ &= \ \conflictSet_2^{a} \cup \conflictSet_2^{b}, 
		    \intertext{where}
		    \conflictSet_2^{a} &= \left\{ (\requestInd_1, \requestInd_2, \workingInd, \protectionInd) \colon \requestInd_1, \requestInd_2 \in \requestSet, \ \{\mssnode_{\requestInd_1}, \mssnode_{\requestInd_2}\} \in E_s, \ \workingInd \in \workingSet^{\requestInd_1}, \  \protectionInd \in \protectionSet^{\requestInd_2} \right\} \\
		    \conflictSet_2^{b} &= \left\{ (\requestInd_2, \requestInd_1, \workingInd, \protectionInd) \colon \requestInd_1, \requestInd_2 \in \requestSet, \ \{\mssnode_{\requestInd_1}, \mssnode_{\requestInd_2}\} \in E_s, \ \workingInd \in \workingSet^{\requestInd_2}, \  \protectionInd \in \protectionSet^{\requestInd_1} \right\}.    
            \intertext{(iii) \ For each edge $ \{\mssnode_{\requestInd_1}, \mssnode_{\requestInd_2}\} \in E_s $, we want the pair of working lightpaths for the associated requests $\requestInd_1$ and $\requestInd_2$ have a link in common. Accordingly, we have }
		    \conflictSet_3 \ &= \ \left\{ (\requestInd_1, \requestInd_2, \workingInd_1, \workingInd_2) \colon \requestInd_1, \requestInd_2 \in \requestSet, \ \{\mssnode_{\requestInd_1}, \mssnode_{\requestInd_2}\} \in E_s, \ \workingInd_1 \in \workingSet^{\requestInd_1}, \  \workingInd_2 \in \workingSet^{\requestInd_2} \right\}. 
		    \intertext{(iv) \ Finally, for each edge $ \{\mssnode_{\requestInd_1}, \mssnode_{\requestInd_2}\} \in E_s $, we want the pair of protection lightpaths for the associated requests $\requestInd_1$ and $\requestInd_2$ to share a link. So, we have}
		    \conflictSet_4 \ &= \ \left\{ (\requestInd_1, \requestInd_2, \protectionInd_1, \protectionInd_2) \colon \requestInd_1, \requestInd_2 \in \requestSet, \ \{\mssnode_{\requestInd_1}, \mssnode_{\requestInd_2}\} \in E_s, \ \protectionInd_1 \in \protectionSet^{\requestInd_1}, \  \protectionInd_2 \in \protectionSet^{\requestInd_2} \right\}. 
		\end{align*}
        
        \medskip
        We note that the total number of tuples in the conflict sets is $\bigO \left( |E_s| \right)$, because we generate a constant number of tuples for each edge in $E_s$. 
        
        In the example provided in Figure \ref{fig:nphard_mss_to_init_paths}, let $\workingInd^{\requestInd} \in \workingSet^{\requestInd}$ and $\protectionInd^{\requestInd} \in \protectionSet^{\requestInd}$ denote the working and protection lightpaths for request $\requestInd \in \{1,2,3\}$. Then, the conflict sets $\exConfSet_i$'s based on the {\mss} instance $\exMSSinst$ in Figure \ref{fig:nphard_mss_to_init_paths} are as follows:
        \begin{align*}
            \exConfSet_1 \ &= \ \varnothing \\
            \exConfSet_2 \ &= \ \left\{ \left(1,2,\workingInd^1,\protectionInd^2 \right), \ \left(2,1,\workingInd^2,\protectionInd^1 \right), \ \left(2,3,\workingInd^2,\protectionInd^3 \right), \ \left(3,2,\workingInd^3,\protectionInd^2 \right) \right\}\\
            \exConfSet_3 \ &= \ \left\{ \left(1,2,\workingInd^1,\workingInd^2 \right), \ \left(2,3,\workingInd^2,\workingInd^3 \right) \right\}\\
            \exConfSet_4 \ &= \ \left\{ \left(1,2,\protectionInd^1,\protectionInd^2 \right), \ \left(2,3,\protectionInd^2,\protectionInd^3 \right) \right\}.
        \end{align*}

		\item {\it Network and lightpath set modification.} The initial form of the {\PRWAreq} network $G$ consists of the disjoint union of cycles (illustrated in part (ii) of Figure \ref{fig:nphard_mss_to_init_paths}), and hence does not yet incorporate the common links of lightpaths expressed in conflict sets $\conflictSet_1, \ldots, \conflictSet_4$.
		In this last step of the construction of the {\PRWAreq} instance $\PRWAreqInst$, we modify $G$ to make it consistent with the conflict sets specified above.
		To this end, we iterate over the tuples in the conflict sets, and modify and re-route the associated lightpaths so that they share a link.
		In alignment with the above-defined conflict sets, we keep the working and protection lightpaths of any fixed request link-disjoint, and modify each lightpath pairs appearing in the conflict sets such that they have one link in common.
		For ease of exposition, we use our running example in explaining the procedure.
		Let us consider the conflict tuple $\left(1,2,\workingInd^1,\workingInd^2 \right)$, where $\left((s^1, \wnode^1_1), (\wnode^1_1, t^1) \right)$ and $\left((s^2, \wnode^2_1), (\wnode^2_1, t^2) \right)$ are the ordered set of links constituting the working lightpaths $\workingInd^1$ and $\workingInd^2$, respectively.
		We arbitrarily choose the lightpath $\workingInd^1$ to be the ``host" to intersect $\workingInd^1$ and $\workingInd^2$ on, i.e., to create a shared link that currently belongs to the host.
		Starting with the reconstruction of the lightpath $\workingInd^1$, we delete the link $(\wnode^1_1, t^1)$, create a node labelled $\wnode^1_2$, and add links $(\wnode^1_1, \wnode^1_2)$ and $(\wnode^1_2, t^1)$, so that $\left((s^1, \wnode^1_1), (\wnode^1_1, \wnode^1_2), (\wnode^1_2, t^1) \right)$ becomes the revised working lightpath for request $1$.
		Then, for lightpath $\workingInd^2$, we remove the link $(\wnode^2_1, t^2)$, create a node labelled $\wnode^2_2$, and add links $(\wnode^2_1, \wnode^1_1)$, $(\wnode^1_2, \wnode^2_2)$ and $(\wnode^2_2, t^2)$, as a result of which the working lightpath for request $2$ becomes $\left((s^2, \wnode^2_1), (\wnode^2_1, \wnode^1_1), (\wnode^1_1, \wnode^1_2), (\wnode^2_2, t^2) \right)$, as highlighted in part (ii) of Figure \ref{fig:nphard_path_constr}, with $(\wnode^1_1, \wnode^1_2)$ now being the common link of the modified lightpaths $\workingInd^1$ and $\workingInd^2$.
		As an additional illustration of our lightpath reconstruction procedure, we next consider the conflict tuple $\left(1,2,\workingInd^1,\protectionInd^2 \right)$ for our running example, arbitrarily select the host to be the lightpath $\workingInd^1$, and then perform a similar set of operations to re-route the lightpaths, the result of which is shown in part (iii) of Figure \ref{fig:nphard_path_constr}. 
		So, for each conflict tuple under consideration, we select one of the lightpaths to be the host, extend it by one link, and redirect the other lightpath through the newly created link.
        We note that this procedure establishes each lightpath intersection through newly created nodes and links in $G$, and thereby keeps the remaining lightpaths intact at every step.

        \begin{figure}[!h]
    \centering
     \tikzset{main_node/.style={circle,draw,line width=0.75pt,minimum size=0.75cm,inner sep=0pt},}
    \tikzset{node1/.style={circle,draw,line width=0.75pt,minimum size=0.75cm,inner sep=0pt,fill = red},}
    \tikzset{node2/.style={circle,draw,line width=0.75pt,minimum size=0.75cm,inner sep=0pt,fill = yellow},}
    \tikzset{node3/.style={circle,draw,line width=0.75pt,minimum size=0.75cm,inner sep=0pt,fill = blue!60!white},}
    \tikzstyle{edge} = [draw,line width=1pt,-]
    \tikzstyle{arclr} = [draw,line width=1pt,-latex]
    \tikzstyle{arcrl} = [draw,line width=1pt,latex-]
    \tikzstyle{highlighted edge} = [draw,line width=6pt,-,green!70!black, opacity=0.3]
    
    \scalebox{0.59}{
    \begin{tikzpicture}
    
        \node[node1] (s_1) {$s^1$};
        \node[main_node] (w_11) [above right = 0.25 cm and 1.25 cm of s_1] {$\wnode^1_1$};
        \node[main_node] (p_11) [below right = 0.25 cm and 1.25 cm of s_1]  {$\pnode^1_1$};
        \node[node1] (t_1) [below right = 0.25 cm and 1.25 cm of w_11]  {$t^1$};
        
         \draw[arclr] (s_1) edge[bend left=20]  (w_11);
         \draw[arclr] (w_11) edge[bend left=20] (t_1);
         \draw[arclr] (s_1) edge[bend right=20]  (p_11);
         \draw[arclr] (p_11) edge[bend right=20] (t_1);
             
        \begin{scope}[yshift=-3.25cm]
            \node[node2] (s_2) {$s^c$};
            \node[main_node] (w_21) [above right = 0.25 cm and 1.25 cm of s_2] { $\wnode^2_1$};
            \node[main_node] (p_21) [below right = 0.25 cm and 1.25 cm of s_2]  {$\pnode^2_1$};
            \node[node2] (t_2) [below right = 0.25 cm and 1.25 cm of w_21]  {$t^2$};
            
            \node (info1) [below = 0.5 cm of p_21]  {(i)};

        \draw[arclr] (s_2) edge[bend left=20]  (w_21);
     	\draw[arclr] (w_21) edge[bend left=20] (t_2);
     	\draw[arclr] (s_2) edge[bend right=20]  (p_21);
     	\draw[arclr] (p_21) edge[bend right=20] (t_2); 
                 
        \end{scope}
        
        \begin{scope}[xshift=6.5cm]
            \node[node1] (s_1) {$s^1$};
            \node[main_node] (w_11) [above right = 0.25 cm and 1.25 cm of s_1] {$\wnode^1_1$};
            \node[main_node] (w_12) [right = 1.15 cm of w_11] {$\wnode^1_2$};
            \node[main_node] (p_11) [below right = 0.25 cm and 1.25 cm of s_1]  {$\pnode^1_1$};
            \node[node1] (t_1) [below right = 0.25 cm and 1.25 cm of w_12]  {$t^1$};
            
            \node (trans_node_1) [below left = 1 cm and 1.25 cm of s_1] {};
            \node (trans_node_2) [below left = 1 cm and 0.5 cm of s_1] {};
            \draw[ultra thick,->] (trans_node_1) to (trans_node_2);

         \draw[arclr] (s_1) edge[bend left=20]  (w_11);
         \draw[arclr] (w_11) edge[bend left=5] (w_12);
         \draw[arclr] (w_12) edge[bend left=20] (t_1);
         \draw[arclr] (s_1) edge[bend right=20]  (p_11);
         \draw[arclr] (p_11) edge[bend right=20] (t_1);

        \end{scope}
        
        \begin{scope}[xshift=6.5cm, yshift=-3.25cm]
            \node[node2] (s_2) {$s^2$};
            \node[main_node] (w_21) [above right = 0.25 cm and 1.25 cm of s_2] { $\wnode^2_1$};
            \node[main_node] (w_22) [right = 1.15 cm of w_21] { $\wnode^2_2$};
            \node[main_node] (p_21) [below right = 0.25 cm and 1.25 cm of s_2]  {$\pnode^2_1$};
            \node[node2] (t_2) [below right = 0.25 cm and 1.25 cm of w_22]  {$t^2$};

            \node (info1) [below right = 0.5 cm and 0.25 cm of p_21]  {(ii)};

                \draw[arclr] (s_2) edge[bend left=20]  (w_21);
                \draw[arclr] (w_21) edge[out=30,in=-80] (w_11);
                \draw[arcrl] (w_22) edge[out=150,in=-80] (w_12);
                \draw[arclr] (w_22) edge[bend left=20] (t_2);                    
                \draw[arclr] (s_2) edge[bend right=20]  (p_21);
                \draw[arclr] (p_21) edge[bend right=20] (t_2);
               
            \path[highlighted edge]
                (s_2)  to[bend left=20]  (w_21)
                (w_11) to[bend left=5] (w_12)
                (w_21) to[out=30,in=-80] (w_11)
                (w_12) to[out=-80,in=150] (w_22)
                (w_22) to[bend left=20] (t_2); 
            \path[edge]
                (w_11) edge[line width = 2.5pt, bend left=5] (w_12);
        \end{scope}
        
        \begin{scope}[xshift=14.5cm]
            \node[node1] (s_1) {$s^1$};
            \node[main_node] (w_11) [above right = 0.25 cm and 1.25 cm of s_1] {$\wnode^1_1$};
            \node[main_node] (w_12) [right = 1.15 cm of w_11] {$\wnode^1_2$};
            \node[main_node] (w_13) [right = 1.15 cm of w_12] {$\wnode^1_3$};
            \node[main_node] (p_11) [below right = 0.25 cm and 1.25 cm of s_1]  {$\pnode^1_1$};
            \node[node1] (t_1) [below right = 0.25 cm and 1.25 cm of w_13]  {$t^1$};
            
            \node (trans_node_1) [below left = 1 cm and 1.25 cm of s_1] {};
            \node (trans_node_2) [below left = 1 cm and 0.5 cm of s_1] {};
            \draw[ultra thick,->] (trans_node_1) to (trans_node_2);
   
                 \draw[arclr] (s_1) edge[bend left=20]  (w_11);
                 \draw[arclr] (w_11) edge[bend left=5] (w_12);
                 \draw[arclr] (w_12) edge[bend left=5] (w_13);
                 \draw[arclr] (w_13) edge[bend left=20] (t_1);
                 \draw[arclr] (s_1) edge[bend right=20]  (p_11);
                 \draw[arclr] (p_11) edge[bend right=15] (t_1);
                 
        \end{scope}
        
        \begin{scope}[xshift=14.5cm, yshift=-3.25cm]
        
            \node[node2] (s_2) {$s^2$};
            \node[main_node] (w_21) [above right = 0.25 cm and 1.25 cm of s_2] { $\wnode^2_1$};
            \node[main_node] (w_22) [right = 1.15 cm of w_21] { $\wnode^2_2$};
            \node[main_node] (p_21) [below right = 0.25 cm and 1.25 cm of s_2]  {$\pnode^2_1$};
            \node[main_node] (p_22) [right = 1.15 cm of p_21]  {$\pnode^2_2$};
            \node[node2] (t_2) [below right = 0.25 cm and 1.25 cm of w_22]  {$t^2$};
            
            \node (info1) [below = 0.5 cm of p_22]  {(iii)};

                \draw[arclr] (s_2) edge[bend left=20]  (w_21);
                \draw[arclr] (w_21) edge[out=30,in=-80] (w_11);
                \draw[arclr] (w_22) edge[bend left=20] (t_2);
                \draw[arclr] (w_12) edge[out=-80,in=150] (w_22);
                \draw[arclr] (s_2) edge[bend right=20]  (p_21);
                \draw[arclr] (p_21) edge[out=80,in=-50] (w_12);
                \draw[arclr] (w_13) edge[out=-70,in=80] (p_22);
                \draw[arclr] (p_22) edge[bend right=15] (t_2);
                
            \path[highlighted edge]
                (s_2) to[bend right=20]  (p_21)
                (p_21) to[out=80,in=-50] (w_12)
                (w_12) to[bend left=5] (w_13)
                (w_13) to[out=-70,in=80] (p_22)
                (p_22) to[bend right=15] (t_2);

            \path[edge]
                (w_12) edge[line width = 2.5pt, bend left=5] (w_13);
        \end{scope}
    \end{tikzpicture}
    }
    \caption{Illustration of the re-routing of lightpaths in the {\PRWAreq} network $\exPRWAnetwork$. }
    \label{fig:nphard_path_constr}
\end{figure}	

	\end{enumerate}
    
    {\it Reduction complexity.} Let us investigate the complexity of building the {\PRWAreq} instance $\PRWAreqInst =  \left(G, \requestSet, \{\workingSet^\requestInd\}_{r \in \requestSet}, \{ \protectionSet^\requestInd \}_{r \in \requestSet}  \right) $ from a given {\mss} instance $G_s = \left( V_s, E_s \right)$.
    Steps $1$ through $3$ respectively take $O\left(|V_s|\right)$, $O\left(|V_s|\right)$, and $ O\left(|E_s|\right) $ time.
    In step $4$, we delete two links, add three links and two nodes for each conflict tuple, each of which takes constant time. Since there are $O\left(|E_s|\right)$-many tuples, we spend $O\left(|E_s|\right)$ time in total for step $4$.
    Then, the complexity of instance reduction phase amounts to $O\left(|V_s| + |E_s| \right)$ in total, which is polynomial in the size of the {\mss} instance $G_s$. 
    Hence, we have a polynomial time reduction from {\mss} to {\PRWAreq}.

    {\it Problem complexity.} Now, we need to show that solving {\dmss} in $G_s$ is equivalent to solving {\DPRWAreq} in $\PRWAreqInst$. That is, we should show that there exists a stable set of size at least $k$ in $G_s$ if and only if we can satisfy at least $k$ requests in $\PRWAreqInst$.
    
    $(\Rightarrow)$ Suppose that we are given a stable set of size $k$ in $G_s$, say nodes $\{\mssnode_1, \ldots, \mssnode_k\}$ without loss of generality. 
    Take lightpaths $\{\workingInd^1, \protectionInd^1, \ldots, \workingInd^k, \protectionInd^k\}$ in $\PRWAreqInst$.
    If this selection of lightpaths satisfies the constraints of the {\PRWAreq} problem, then it means that we can grant $k$ requests in $\PRWAreqInst$.
    First, since there is only one working lightpath assigned to every request corresponding to nodes $\{\mssnode_1, \ldots, \mssnode_k\}$, the constraint of assigning at most one working lightpath to each request is satisfied.
    %, the constraints  \eqref{eq:w_p_at_most_one} are already satisfied.
    Second, since no pair of nodes in $\{\mssnode_1, \ldots, \mssnode_k\}$ are linked by an edge, the corresponding lightpaths $\{\workingInd^1, \protectionInd^1, \ldots, \workingInd^k, \protectionInd^k\}$ do not share a link either, 
    so the requirement that the working and protection lightpaths of each granted request being link-disjoint is satisfied as well. 
    %so the constraints  \eqref{eq:conflicts1}--\eqref{eq:conflicts4} are satisfied as well.
    %As $\conflictSet_1 = \varnothing$, there is no constraint of type \eqref{eq:conflicts1}.
    Finally, since we add both the working and protection lightpaths for each request $r$ corresponding to the nodes of the given stable set, each granted request is properly assigned the two types of lightpaths.
    %the constraints  \eqref{eq:w_p_equality} are never violated.
    Therefore, we conclude that we can grant $k$ requests in this case.

    $(\Leftarrow)$ Conversely, suppose that we have a solution granting $k$ requests in the {\DPRWAreq} instance $\PRWAreqInst$, say requests $\{1, \ldots, k\}$.
    Take the nodes in $G_s$ that correspond to requests $\{1, \ldots, k\}$.
    Since all the lightpaths use the same wavelength in $G$, the ones associated with requests $\{1, \ldots, k\}$ cannot share a link.
    This implies that the nodes $\{\mssnode_1, \ldots, \mssnode_k\}$ representing requests $\{1, \ldots, k\}$ cannot be connected (i.e., they form a stable set), because if they were, the related lightpaths would have been forced to share a link by our construction.
    
    This shows that solving {\dmss} on $G_s$ is equivalent to solving {\DPRWAreq} in $\PRWAreqInst$ built by a polynomial-time reduction from $G_s$.
    Since {\DPRWAreq} is in NP, as shown in Lemma \ref{lem:np}, we conclude that {\DPRWAreq} is NP-complete, which implies that {\PRWAreq} is NP-hard. \hfill %\Halmos
\end{proof}

\subsection{Proof of Proposition \ref{prop:penalty_coef}}
\label{OS:proof:prop_penalty_coef}
{\it % \begin{proposition}[QUBO penalty selection]
% \label{prop:penalty_coef}
When 
\begin{align}
\penaltyCoef 
\ > \
\beta  (|\requestSet| \ + \ 1) 
\ - \ 
\alpha \left( 1 \ + \  \sum_{\requestInd \in \requestSet} \left(\length^{\requestInd}_{w_{\min}}  + \ \length^{\requestInd}_{p_{\min}} \right) \right),  \tag{\ref{eq:penalty_coef_lb}}  
\end{align}
\eqref{m:QUBO} is an exact QUBO model for the {\PRWA} problem,
%; that is, any optimal solution of the QUBO model is feasible and optimal for the IP model.
where 
$\length^{\requestInd}_{w_{\min}} 
= \ 
\min_{\workingInd \in \workingSet^{\requestInd}} \left \{\length^\requestInd_\workingInd \right \} $
and 
$\length^{\requestInd}_{p_{\min}}
= \ 
\min_{\protectionInd \in \protectionSet^{\requestInd}} \left \{\length^\requestInd_\protectionInd \right \}$.
% \end{proposition}

}
% \begin{proposition}[QUBO penalty selection]
% \label{prop:penalty_coef}
% When 
% \begin{align}
% \penaltyCoef 
% \ > \
% \beta  (|\requestSet| \ + \ 1) 
% \ - \ 
% \alpha \left( 1 \ + \  \sum_{\requestInd \in \requestSet} \left(\length^{\requestInd}_{w_{\min}}  + \ \length^{\requestInd}_{p_{\min}} \right) \right),  \label{eq:penalty_coef_lb}  
% \end{align}
% the QUBO model in \eqref{m:QUBO} is an exact reformulation of the IP formulation in \eqref{m:IP} for the {\PRWA} problem; that is, any optimal solution of the QUBO model is feasible and optimal for the IP model. 
% \begin{comment}
% %previously suggested condition
% \begin{align}
% \penaltyCoef \ > \ \sum_{\requestInd \in \requestSet} \sum_{\workingInd \in \workingSet^{\requestInd}} \left( \beta - \alpha \length^\requestInd_\workingInd \right),  \label{eq:penalty_coef_lb}    
% \end{align}
% \end{comment}
% \end{proposition}
%
\begin{proof}

Given the vector of working and protection lightpath binary decision variables $x$ and $y$, respectively, we introduce $g$ function to compactly rewrite the QUBO objective function expression in \eqref{eq:OF_qubo} as 
$$\obj(x,y) \ + \ \penaltyCoef \ g(x,y).$$ 
That is, $\obj(x,y)$ is the objective expression of the IP as defined in \eqref{eq:OF} ({\IPbase}), while $g(x,y)$ denotes the summation of the penalty terms in the QUBO objective before the common penalty coefficient of $\penaltyCoef$ is applied. 
{\cemph Even though the feasible spaces of {\IPbase} and {\IPstrong} are indeed the same, the IP model we actually refer to in this proof is {\IPbase} because our QUBO model is reformulation of that model.}
Note that $g(x,y) = 0$ if $(x,y)$ is IP feasible, and $g(x,y) \geq 1$ otherwise by the construction of the penalty terms as explained before. 

We first derive valid upper and lower bounds on the $\obj$ function value over the set of IP feasible solutions, namely
\begin{align*}
    F_{UB} & \coloneqq \max \left\{ \obj(x,y) \colon x,y \ \text{binary and} \ g(x,y) = 0 \right\} \\
    F_{LB} & \coloneqq \min \left\{ \obj(x,y) \colon x,y \ \text{binary and} \ g(x,y) = 0 \right\}.
\end{align*}
It is easy to observe that $F_{UB} = 0$ since (i) the trivial solution of accepting none of the requests yields $\obj(x=\boldsymbol{0},y=\boldsymbol{0}) = 0$ for the maximization problem, thus $F_{UB} \geq 0$, and (ii) the selection of $ \beta$ and $ \alpha $ values as suggested by Proposition \ref{prop:obj_weights} ensures that every granted request adds a negative term to the $f$ value in the IP objective, hence yields $F_{UB} \leq 0$. 
On the other hand, the optimal value to the lower bound problem is
at least the objective value that is obtained
%attained at $(x=\boldsymbol{1},y=\boldsymbol{1})$, i.e., 
by granting all the requests using the shortest working and protection lightpaths for each request, 
%assigning each a working and protection lightpath,
yielding
\begin{align*}
F_{LB} 
\ \geq \ 
%\obj(x=\boldsymbol{1},y=\boldsymbol{1}) 
%\ &= \
\alpha \ \sum_{\requestInd \in \requestSet} \left( \min_{\workingInd \in \workingSet^{\requestInd}} \left \{\length^\requestInd_\workingInd \right \} 
\ + \ 
\min_{\protectionInd \in \protectionSet^{\requestInd}} \left \{\length^\requestInd_\protectionInd \right \}
 \right) 
\ - \ 
\beta \sum_{\requestInd \in \requestSet} 1 %\\[0.2cm]
& \ = \ 
\alpha \ \sum_{\requestInd \in \requestSet} \left(\length^{\requestInd}_{w_{\min}}  + \ \length^{\requestInd}_{p_{\min}} \right)
\ - \
\beta \ |\requestSet|.  
\end{align*}
This is again due to the way we set the values of $\alpha$ and $\beta$ by Proposition \ref{prop:obj_weights}.

Next, we observe that we can separate the range of the QUBO objective values of IP infeasible solutions from those of IP feasible solutions by pushing them beyond $F_{UB}$.
%, which can be achieved by adding a penalty term that is strictly larger than $F_{UB} - F_{LB}$.
As mentioned before, when a binary solution does not satisfy some IP constraints in \eqref{eq:w_p_equality}--\eqref{eq:conflicts4}, the magnitude of violation is at least one.
Due to the way $\alpha$ and $\beta$ values are set by Proposition \ref{prop:obj_weights}, one additional unit of violation can lead to a decrease of the $\obj$ value by at most $\beta - \alpha$, which happens when a zero value in the $x$ vector is changed to one and $y$ vector is kept intact, where the $-\alpha$ term is due to the resulting additional link usage being at least one. 
In order to guarantee that the QUBO objective values are shifted strictly above $F_{UB}$ through penalty terms in case at least one of the {\crev IP constraints in \eqref{eq:w_p_equality}--\eqref{eq:conflicts4}} are violated, we set
\begin{align*}
\penaltyCoef 
\ > \ 
\beta  (|\requestSet| \ + \ 1) - \alpha  \left( 1 \ + \  \sum_{\requestInd \in \requestSet} \left(\length^{\requestInd}_{w_{\min}}  + \ \length^{\requestInd}_{p_{\min}} \right) \right)
\ \geq \
F_{UB} - F_{LB} + \beta - \alpha.
\end{align*}
%so that the QUBO objective values are shifted strictly above $F_{UB}$ through penalty terms in case at least one of the original constraints {\crev in \eqref{eq:w_p_equality}--\eqref{eq:conflicts4}} is violated. 
By selecting the penalty coefficient $\penaltyCoef$ {\crev such that it satisfies the above strict inequality, which is indeed the suggested one in \eqref{eq:penalty_coef_lb}}, we obtain the following relations:
\begin{alignat*}{2}
    & (x,y) \ \text{is feasible to the IP} \ && \implies \ g(x,y) = 0 \\ & &&  \implies \ \obj(x,y) + \penaltyCoef \ g(x,y) \leq F_{UB} \\
    & (x,y) \ \text{is infeasible to the IP} \ &&  \implies g(x,y) \geq 1 \\ & &&  \implies \ \obj(x,y) + \penaltyCoef \ g(x,y) > \obj(x,y) + F_{UB} - F_{LB} + \beta - \alpha \geq F_{UB}. 
\end{alignat*}
As such, the IP feasible solutions will always deliver smaller QUBO objective values than IP infeasible ones. Since our QUBO model is in minimization form, any optimal solution to the QUBO model must be one that is optimal (and feasible) for the IP model. \hfill %\Halmos
\end{proof}

%\end{appendix}

\end{document}